\documentclass[11pt]{article}

\usepackage{amssymb,amscd,amsthm,amsmath}
\usepackage[dvipdfm,colorlinks,breaklinks,
pdfstartview={FitH -32768},
bookmarks,bookmarksnumbered,bookmarksopen,
pdftitle={Weak Homotopy Equivalence Type Result},
pdfkeywords={2000 Mathematics Subject Classification: 46L55.},
pdfauthor={Masaki Izumi and Hiroki Matui}]{hyperref}

\newtheorem{theorem}{Theorem}[section]
\newtheorem{lemma}[theorem]{Lemma}
\newtheorem{proposition}[theorem]{Proposition}
\newtheorem{cor}[theorem]{Corollary}

\theoremstyle{definition}
\newtheorem{definition}[theorem]{Definition}
\newtheorem{example}[theorem]{Example}

\theoremstyle{remark}
\newtheorem{remark}[theorem]{Remark}

\theoremstyle{conjecture}

\theoremstyle{problem}
\newtheorem{problem}[theorem]{Problem}

\numberwithin{equation}{section}

\newcommand{\C}{\mathbb{C}}

\newcommand{\K}{\mathbb{K}}

\newcommand{\N}{\mathbb{N}}

\newcommand{\T}{\mathbb{T}}
\newcommand{\Z}{\mathbb{Z}}

\newcommand{\cO}{\mathcal{O}}

\newcommand{\cN}{\mathcal{N}}

\newcommand\Aut{\operatorname{Aut}}

\newcommand\id{\operatorname{id}}
\newcommand\Ad{\operatorname{Ad}}
\newcommand\Hom{\operatorname{Hom}}
\newcommand\Map{\operatorname{Map}}

\newcommand\Pext{\operatorname{Pext}}

\newcommand{\Lip}{\operatorname{Lip}}

\title{A weak homotopy equivalence type result\\
related to Kirchberg algebras}

\author{Masaki Izumi
\thanks{Supported in part by JSPS KAKENHI Grant Number JP15H03623}\\
Graduate School of Science \\
Kyoto University \\
Sakyo-ku, Kyoto 606-8502, Japan 
\and
Hiroki Matui 
\thanks{Supported in part by JSPS KAKENHI Grant Number JP18K03321} \\
Graduate School of Science \\
Chiba University \\
Inage-ku, Chiba 263-8522, Japan}

\begin{document} 
\maketitle
\begin{abstract} We obtain a weak homotopy equivalence type result between two topological groups associated with 
a Kirchberg algebra: the unitary group of the continuous asymptotic centralizer and the loop group of the automorphism group 
of the stabilization. 
This result plays a crucial role in our subsequent work on the classification of poly-$\mathbb{Z}$ group actions on Kirchberg algebras. 
As a special case, we show that the $K$-groups of the continuous asymptotic centralizer are isomorphic to 
the $KK$-groups of the Kirchberg algebra. 
\end{abstract}
\section{Introduction} 
Purely infinite separable simple nuclear $C^*$-algebras, now called Kirchberg algebras, form one of the most prominent 
classes of classifiable $C^*$-algebras, and the celebrated Kirchberg-Phillips classification theorem says that 
their Morita equivalence classes are completely determined by $KK$-equivalence (see \cite{P00}, \cite{R2002}). 
Based on the classification theorem, Dadarlat \cite{D} identified the homotopy groups of their automorphism groups 
with relevant $KK$-groups, which led to subsequent work of Dadarlat-Pennig \cite{DP15}, \cite{DP-II}, \cite{DP-I} 
on novel and beautiful interplay between algebraic topology and the classification theory of $C^*$-algebras. 
In this paper we further investigate topological features of the automorphism groups of Kirchberg algebras, 
motivated by our ongoing project of the classification of discrete amenable group actions on Kirchberg algebras 
\cite{IM2010}, \cite{IM1}, \cite{IM2}. 

For two topological spaces $X$ and $Y$, a continuous map $f:X\to Y$ is said to be a weak homotopy equivalence 
if it induces isomorphisms of the homotopy groups $f_*:\pi_i(X,x)\to \pi_i(Y,f(x))$ for any $i$ and for any $x\in X$. 
The main purpose of this paper is to show a weak homotopy equivalence type result between the two topological groups 
associated with a unital Kirchberg algebra $A$: 
the unitary group $U(A_\flat)$ of the continuous asymptotic centralizer algebra 
$$A_\flat:=(C^b([0,1),A)/C_0([0,1),A))\cap A',$$
and the loop group $\Omega\Aut(A\otimes \K)$ of the automorphism group of the stabilization $A\otimes \K$ of $A$. 
More precisely, for a compact metrizable space $X$ we construct a group homomorphism 
$$\Pi_{A,X}:[X,U(A_\flat)]\to [X,\Omega \Aut(A\otimes \K)],$$
which is natural in $X$, and is an isomorphism if $X$ is the sphere $S^i$ (Theorem \ref{main}).  
Our construction does not show that the map $\Pi_{A,X}$ is induced by a continuous map, and 
we have not established a genuine weak homotopy equivalence. 

Our main result Theorem \ref{main} is essential in our analysis \cite{IM2} of poly-$\Z$ group actions on Kirchberg algebras. 
Generally speaking, it is much easier to handle unitary groups than automorphism groups for $C^*$-algebras;
the natural topology of the unitary group of a unital $C^*$-algebra is given by the operator norm, which allows simple 
perturbation arguments via functional calculus, while the point norm topology of the automorphism group is much less flexible. 
Thanks to Theorem \ref{main}, we can transfer topological information of $\Aut(A\otimes \K)$ to that of 
$U(A_\flat)$, and perform analysis on the latter. 

After constructing the map $\Pi_{A,X}$ and stating the main result and its direct consequences in Section 2, 
we prove Theorem \ref{main} in Section 3. 
The first step is to show that the $K$-groups $K_*(A_\flat)$ of $A_\flat$ are isomorphic to the $KK$-groups $KK(A,S^*A)$ 
for $*=0,1$, by using asymptotic morphisms arising from projections and unitaries in $A_\flat$. 
Probably this kind of idea goes back to Kirchberg's work \cite{Kir}, and has been shared among specialists as a folklore. 
In Section 4, we apply the same idea to the central sequence algebras to compute their $K$-groups too. 

Nakamura's homotopy theorem \cite[Theorem 7]{N} is one of the most powerful technical tools when we work on the asymptotic centralizers 
of Kirchberg algebras. 
In Appendix, we prove a refinement of it, Theorem \ref{refinedNakamura}, which is used in the proof of our main result as well as 
in our companion papers \cite{IM1}, \cite{IM2}.

\bigskip

\textbf{Acknowledgement. }
Masaki Izumi would like to thank 
Isaac Newton Institute for Mathematical Sciences 
for its hospitality. 
Masaki Izumi and Hiroki Matui thank the referee 
for careful reading and for giving several useful comments 
that made this manuscript much more accessible to readers.

\section{Construction of $\Pi_{A,X}$} 
For a unital $C^*$-algebra $B$, we denote by $U(B)$ the unitary group of $B$. 
When $B$ is non-unital, we define $U(B)$ to be the set of unitaries $u$ in the unitization of $B$ 
satisfying $u-1\in B$. 
We denote by $\K$ the set of compact operators on a separable infinite dimensional Hilbert space. 
The automorphism group $\Aut(B\otimes \K)$ of $B\otimes \K$ is a topological group 
equipped with the point norm topology, and it is a Polish group if $B$ is separable.  
We set 
$$B^\infty=\ell^\infty(\N,B)/c_0(\N,B).$$
Then $B$ is embedded into $B^\infty$ as the image of the constant sequences, and we denote $B_\infty=B^\infty\cap B'$. 
For a continuous version, we use the following shorthand notation: 
$$C^bB=C^b([0,1),B),\quad C_0B=C_0([0,1),B),$$ 
$$B^\flat=C^bB/C_0B,\quad B_\flat=B^\flat\cap B'.$$ 
We denote by $q_B$ the quotient map $q_B:C^bB\to B^\flat$, 
though we often omit it if there is no possibility of confusion. 
We denote by $SB$ the suspension of $B$, that is, $SB=B\otimes C_0(0,1)$. 

A based space $(X,x_0)$ is a topological space $X$ with a fixed base point $x_0\in X$. 
When $X$ is a compact Hausdorff space, we denote 
$C(X,x_0)=C_0(X\setminus \{x_0\})$, which is the set of continuous functions on $X$ vanishing at $x_0$. 
We denote by $I$ the closed unit interval $[0,1]$, and the circle $S^1$ is often identified with $I/\partial I$. 
The reduced suspension $\Sigma X$ is defined by  
$$X\times I/((X\times \partial I)\cup (\{x_0\}\times I)),$$
whose base point is $\Sigma x_0=(X\times \partial I)\cup (\{x_0\}\times I)\in \Sigma X$. 
Then we can identify $C(\Sigma X,\Sigma x_0)$ with $SC(X,x_0)$. 
A topological group is always regarded as a based space with a base point $e$, the neutral element.  

For based spaces $(X,x_0)$ and $(Y,y_0)$, a based map $f:X\to Y$ is a continuous map satisfying $f(x_0)=y_0$. 
We denote by $\Map(X,Y)$ the set of based maps from $X$ to $Y$, and by $[X,Y]$ the set of based homotopy classes 
of based maps from $X$ to $Y$. 
To introduce topology in $\Map(X,Y)$, we follow the standard procedure in algebraic topology (see 
\cite[Chapter 6]{DK} for example). 
In this paper we treat only the case where $X$ is a compact metrizable space and $Y$ is a metrizable space, 
and the topology of $\Map(X,Y)$ in this case is nothing but the usual compact open topology. 

We denote by $\Omega X$ the based loop space of $X$, that is, 
$$\Omega X=\{f\in \Map(I,X);\; f(0)=f(1)=x_0\},$$
whose base point is a constant loop. 

Let $(X,x_0)$ be a pointed compact Hausdorff space, and let $B$ be a unital $C^*$-algebra. 
Then $f\in \Map(X,U(B))$ is regarded as an element of $U(C(X,B))$ whose image under the 
evaluation map at $x_0$ is $1_B$, the unit of $B$. 
Thus the map associating the $K_1$-class of $f$ in $K_1(C(X)\otimes B)$ to the homotopy class $[f]$ of $f$ in $[X,U(B)]$ is 
a natural homomorphism from $[X,U(B)]$ to $K_1(C(X,x_0)\otimes B)$. 
The next lemma follows from \cite[Lemma 2.4]{BRR}. 

\begin{lemma}\label{identification} Let $B$ be a unital $C^*$-algebra satisfying the following condition: 
for any separable $C^*$-subalgebra $D\subset B$ there exists a unital embedding of the Cuntz algebra $\cO_\infty$ into $B^\infty\cap D'$. 
Then for any based compact metric space $(X,x_0)$,   
the natural map from $[X,U(B)]$ to $K_1(C(X,x_0)\otimes B)$ is an isomorphism. 
\end{lemma}

In what follows, we assume that $A$ is a unital Kirchberg algebra and $(X,x_0)$ is a based compact metric space.  
Then we can apply the above lemma to $B=A_\flat$ (see \cite[Lemma 5.3]{IM1} for example), and we identify 
$[X,U(A_\flat)]$ with $K_1(C(X,x_0)\otimes A_\flat)$. 
In particular, we have $\pi_0(U(A_\flat))=K_1(A_\flat)$, and for $n>0$, 
$$\pi_n(U(A_\flat))\cong \left\{
\begin{array}{ll}
K_0(A_\flat) , &\quad \textrm{if $n$ is odd}, \\
K_1(A_\flat) , &\quad \textrm{if $n$ is even},
\end{array}
\right.
$$ by Bott periodicity.   

We now construct a group homomorphism $\Pi_{A,X}:[X,U(A_\flat)]\to [X,\Omega\Aut(A\otimes \K)]$.

\begin{lemma}\label{lift} Let $A$ be a unital $C^*$-algebra, let $(X,x_0)$ be a based compact Hausdorff space,  
and let $u\in \Map(X,U(A_\flat))$. 
\begin{itemize}
\item [$(1)$] There exists a lifting $\tilde{u}\in \Map(X,U(C^bA))$ satisfying $q_A(\tilde{u}(x))=u(x)$ for all $x\in X$. 
\item [$(2)$] Treating $\tilde{u}$ in $(1)$ as a function of two variables $(x,t)\in X\times [0,1)$, we have 
$$\lim_{t\to1}\sup_{x\in X}\|a\tilde{u}(x,t)-\tilde{u}(x,t)a\|=0,\quad \forall a\in A.$$
\item [$(3)$] If $\tilde{u}'\in \Map(X,U(C^bA))$ is another lifting of $u$, then 
$$\lim_{t\to1}\sup_{x\in X}\|\tilde{u}(x,t)-\tilde{u}'(x,t)\|=0.$$
In particular $u$ determines the unique element $q_{C(X,A)}(\tilde{u})\in C(X,A)^\flat\cap A'$. 
\end{itemize}
\end{lemma}

\begin{proof} (1) 
Thanks to the Bartle-Graves selection theorem \cite{BG}, there exists a continuous lifting $\tilde{u}\in C(X,C^bA)$ 
with $q_A(\tilde{u}(x))=u(x)$ for all $x\in X$. 
We treat $\tilde{u}$ as a function with two variables $(x,t)\in X\times [0,1)$. 

For each $x\in X$, we have 
$$\lim_{t\to 1}\|\tilde{u}(x,t)^*\tilde{u}(x,t)-1\|=\lim_{t\to 1} \|\tilde{u}(x,t)\tilde{u}(x,t)^*-1\|=0,$$
and there exits $t_x\in [0,1)$ satisfying  
$$\sup_{t\in [t_x,1)}\|\tilde{u}(x,t)^*\tilde{u}(x,t)-1\|<1/2,\quad \sup_{t\in [t_x,1)}\|\tilde{u}(x,t)\tilde{u}(x,t)^*-1\|<1/2.$$
Since  
$$\lim_{z\to x}\sup_{t\in [0,1)}\|\tilde{u}(z,t)-\tilde{u}(x,t)\|=0,$$
there exists a neighbourhood $U_x$ of $x$ in $X$ satisfying 
$$\|\tilde{u}(y,t)^*\tilde{u}(y,t)-1\|<1/2,\quad \|\tilde{u}(y,t)\tilde{u}(y,t)^*-1\|<1/2,$$
for all $y\in U_x$ and all $t_x \leq t<1$.  
Since $X$ is compact, this implies that there exists $0 \leq t_0<1$ satisfying
$$\|\tilde{u}(x,t)^*\tilde{u}(x,t)-1\|<1/2,\quad \|\tilde{u}(x,t)\tilde{u}(x,t)^*-1\|<1/2,$$
for all $t_0\leq t<1$ and all $x\in X$. 
By replacing $\tilde{u}(x,t)$ with $\tilde{u}(x,t)|\tilde{u}(x,t)|^{-1}$ for $t_0\leq t<1$, $x\in X$, 
we may assume $\tilde{u}\in U(C(X,C^bA))$. 
Since $q_A(\tilde{u}(x_0,\cdot))=u(x_0)=1$, we have 
$$\lim_{t\to1}\|\tilde{u}(x_0,t)-1\|=0.$$
Thus by replacing $\tilde{u}(x,t)$ with  $\tilde{u}(x,t)\tilde{u}(x_0,t)^*$, we obtain the desired lifting of $u$.

(2) For $a\in A$ we set $f_t(x):=\|a\tilde{u}(x,t)-\tilde{u}(x,t)a\|$. 
Since $a\tilde{u}-\tilde{u}a\in C(X,C^b A)$ and $q_A(a\tilde{u}(x)-\tilde{u}(x)a)=0$, 
the family $\{f_t\}_{t\in [0,1)}$ of continuous functions on $X$ is equicontinuous, which converges to 0 pointwisely. 
Thus the convergence is uniform, and we get the statement.  

(3) The statement follows from a similar argument as in the proof of (2). 
\end{proof}

Since the unitary group $U(M(A\otimes \K))$ of the stable multiplier algebra $M(A\otimes \K)$ is contractible in norm topology 
(see \cite[Theorem 16.8]{W-O}), 
we can choose a norm continuous map $U$ from $X\times [0,1)$ to $U(M(A\otimes \K))$ satisfying 
$U(x,t)=\tilde{u}(x,t)\otimes 1_{\K}$ for all $1/2\leq t<1$, $x\in X$, $U(x_0,t)=1$ for all $t\in [0,1)$, 
and $U(x,0)=1$ for all $x\in X$. 
For $a\in A\otimes \K$, we set  
$$\rho^u_{(x,t)}(a)=U(x,t)aU(x,t)^*,$$ 
for $x\in X$, $t\in [0,1)$, and  
we set $\rho^u_{(x,1)}(a)=a$ for $x\in X$. 
Thanks to Lemma \ref{lift},(2), we see that $\rho^u$ belongs to 
$\Map(\Sigma X, \Aut(A\otimes \K))$. 
We often identify the two spaces $\Map(\Sigma X,\Aut(A\otimes \K))$ and $\Map(X,\Omega\Aut(A\otimes \K))$.

\begin{lemma} \label{well-defined}
The homotopy class $[\rho^u]\in [X,\Omega\Aut(A\otimes \K)]$ does not depend on 
either the choice of $\tilde{u}$ or that of $U$, and it depends only on the homotopy class 
$[u]\in [X,U(A_\flat)]$. 
\end{lemma}

\begin{proof} 
Suppose that we make different choices and get $\tilde{u}'$, $U'$ and ${\rho^u}'$ instead of 
$u$, $U$ and $\rho^u$ respectively. 
Let $V(x,t)=U'(x,t)U(x,t)^*$, which is a continuous map from $X\times [0,1)$ to 
$U(M(A\otimes \K))$ equipped with the norm topology. 
Then we have $V(x,0)=1$ and 
$$\sup_{x\in X}\|V(x,t)-1\|=\sup_{x\in X}\|\tilde{u}(x,t)-\tilde{u}'(x,t)\|$$
for $1/2\leq t<1$. 
Thanks to Lemma \ref{lift},(3), we may regard $V$ as an element in 
$$\Map(\Sigma X,U(M(A\otimes \K)),$$ 
and ${\rho^u}'=\Ad V\circ \rho^u$. 
Since $U(M(A\otimes \K))$ is contractible, the two elements $\rho^u$ and ${\rho^u}'$ are homotopic in 
$\Map(\Sigma X, \Aut(A\otimes \K))$, and the map 
$$\Map(X, U(A_\flat))\ni u\mapsto [\rho^u]\in [\Sigma X,\Aut(A\otimes \K)],$$ 
does not depend on either the choice of $\tilde{u}$ or that of $U$. 

Repeating the same argument with $[0,1]\times X$ instead of $X$, we see that the homotopy class 
$[\rho^u]$ depends only on the homotopy class $[u]\in [X,U(A_\flat)]$. 
\end{proof}

\begin{definition} Let $A$ be a unital Kirchberg algebra, and let $(X,x_0)$ be a pointed compact metrizable space. 
We denote by $\Pi_{A,X}$ the map 
$$\Pi_{A,X}:[X,U(A_\flat)]\ni [u]\mapsto [\rho^u]\in [X,\Omega\Aut(A\otimes \K)],$$
constructed as above. 
When $X$ is the sphere $S^n$ for $n=0,1,\ldots,$ we denote $\Pi_{A,n}=\Pi_{A,S^n}$, which is a map  
from $\pi_n(U(A_\flat))$ to $\pi_n(\Omega\Aut(A\otimes \K))=\pi_{n+1}(\Aut(A\otimes \K))$. 
\end{definition}

\begin{remark}
We can immediately see from the construction that $\Pi_{A,X}$ is a group homomorphism, and it is natural in $X$. 
\end{remark}

Now we state the main theorem of this paper. 

\begin{theorem}\label{main} Let $A$ be a unital Kirchberg algebra. 
Then $\Pi_{A,n}$ is an isomorphism from $\pi_n(U(A_\flat))$ to $\pi_n(\Omega\Aut(A\otimes \K))=\pi_{n+1}(\Aut(A\otimes \K))$ 
for any non-negative integer $n$. 
\end{theorem}

\begin{remark}\label{Dadarlat}
Dadarlat \cite[Corollary 5.11]{D} showed that $\pi_{n+1}(\Aut(A\otimes \K))$ is isomorphic to $KK(A,S^{n+1}A)$. 
Since we need his map, 
$$\overline{\chi}:[X,\Aut(A\otimes \K)]\rightarrow KK(A,C(X,x_0)\otimes A),$$
we recall its definition here. 
For a $C^*$-algebra $B$ and a compact Hausdorff space $Y$, we denote by $j_{B,Y}$ the homomorphism from 
$B$ to $C(Y)\otimes B$ given by $j_{B,Y}(b)= 1\otimes b$. 
We regard $\alpha\in \Map(X,\Aut(A\otimes \K))$ as a homomorphism from $A$ to $C(X)\otimes A\otimes \K$, giving rise to 
$KK(\alpha)\in KK(A,C(X)\otimes A)$, which depends only on the homotopy class of $\alpha$. 
Since $\alpha_{x_0}=\id_{A\otimes \K}$, the difference $KK(\alpha)-KK(j_{A\otimes \K,X})$ falls in $KK(A,C(X,x_0)\otimes A)$, 
and this is the definition of $\overline{\chi}([\alpha])$. 
Dadarlat \cite[Thorem 5.9]{D}) showed that $\overline{\chi}$ is a bijection whenever the pair $(A,X)$ 
is $KK$-continuous. 
He also showed that if $X$ is an $H'$-space (also called co-$H$-space), it is a group homomorphisms (see \cite[Theorem 4.6]{D}). 
Since $\Sigma X$ is an $H'$-space, the map 
$$\overline{\chi}\circ \Pi_{A,X}:[X,U(A_\flat)]\to KK(A,SC(X,x_0)\otimes A)$$ 
is a group homomorphism.
\end{remark}

Our main theorem together with Dadarlat's result includes the following statement as a special case, 
though we will directly prove it later as we need it for the proof of our main result. 
 
\begin{cor} Let $A$ be a unital Kirchberg algebra. 
Then $K_*(A_\flat)$ is isomorphic to $KK(A,S^*A)$ for $*=0,1$. 
\end{cor}

Since $\Pi_{A,X}$ is natural in $X$, Theorem \ref{main} and the Bott periodicity (see Lemma \ref{Bott})
imply the following statement via the Puppe exact sequence for cofibrations.  

\begin{cor} Let $A$ be a unital Kirchberg algebra, and let $(X,x_0)$ be a pointed finite CW-complex with $x_0$ a 0-cell. 
Then the map $\Pi_{A,X}:[X,U(A_\flat)]\to [X,\Omega\Aut(A\otimes \K)]$ is a group isomorphism. 
\end{cor}

\begin{proof} Note that the inclusion map from any subcomplex into $X$ is a cofibration \cite[Theorem 6.22, 6.23]{DK}. 
Since $\Pi_{A,X}$ is natural in $X$, the statement for $\Sigma^nX$ with $n\geq 1$, instead of $X$, 
follows from an induction argument with respect to 
the number of cells using Theorem \ref{main}, Lemma \ref{Bott} and the Puppe exact sequence for cofibrations \cite[Theorem 6.42]{DK}. 
Now the statement for $X$ follows from Lemma \ref{Bott}. 
\end{proof}

Before starting the proof of Theorem \ref{main}, we observe that the map $\Pi_{A,X}$ depends only on 
the Morita equivalence class of $A$ as long as $A$ is unital. 
Let $p\in A$ be a non-zero projection. 
Then \cite[Theorem 4.23]{Br} implies that there exists an isometry $V\in M(A\otimes \K)$ whose range projection 
is $p\otimes 1$. 
Thus we get an explicit isomorphism $A\otimes \K\ni x\mapsto VxV^* \in pAp\otimes \K$, and a group isomorphism 
$\Theta:\Aut(pAp\otimes \K)\to \Aut(A\otimes \K)$ given by $\Theta(\alpha)(a)=V^*\alpha(VaV^*)V$, which induces 
an isomorphism 
$$\Theta_*:[X,\Omega\Aut(pAp\otimes \K)]\to [X,\Omega \Aut(A\otimes \K)].$$ 

\begin{lemma}\label{Morita} Let the notation be as above. 
\begin{itemize}
\item[$(1)$] The map $\theta:A_\flat \ni x \mapsto px\in (pAp)_\flat$ is an isomorphism. 
\item[$(2)$] $\Pi_{A,X}=\Theta_*\circ \Pi_{pAp,X}\circ \theta_*$ holds, where 
$\theta_*:[X,U(A_\flat)]\to [X,U((pAp)_\flat)]$ 
is the isomorphism induced by $\theta$. 
\end{itemize}
\end{lemma}

\begin{proof} (1) It is clear that $\theta$ is a homomorphism. 
We choose an isometry $v\in A$ with $vv^*\leq p$. 
Let $b\in (pAp)_\flat$. 
Then for any $a\in A$, we get 
$$v^*bva=v^*b(vav^*)v=v^*(vav^*)bv=av^*bv,$$
and $v^*bv\in A_\flat$. 
Thus the inverse of $\theta$ is given by $pA_\flat p\ni b\mapsto v^*bv\in A_\flat$, 
and $\theta$ is an isomorphism. 

(2) For $u\in \Map(X,U(A_\flat))$, let $\tilde{u}$, $U$ and $\rho^u$ be as in the proof of Lemma \ref{lift}. 
(1) implies $pu\in \Map(X,U((pAp)_\flat))$. 
We choose a lifting $\tilde{u}'\in \Map(X,(pAp)^\flat)$ of $pu$, and a continuous map $U':X\times [0,1)\to U(M(pAp\otimes \K))$ 
satisfying $U'(x,t)=\tilde{u}'(x,t)\otimes 1_{\K}$ for any $x\in X$, $1/2\leq t<1$, and $U'(x,0)=p\otimes 1_{\K}$ 
for any $x\in X$. 
We define $$\rho^{pu}\in \Map(X,\Omega\Aut(pAp\otimes \K))$$ as before using $U'$. 
To prove the statement it suffices to show 
$$[\Theta(\rho^{pu})]=[\rho^u]\in [X,\Omega \Aut(A\otimes \K)].$$  

For $a\in A\otimes \K$ and $(x,t)\in X\times [0,1)$, we have 
$$\Theta(\rho^{pu}_{x,t})(a)=\Ad (V^*U'(x,t)V)(a)=\Ad W(x,t)\circ \rho^u_{x,t}(a),$$
where $W(x,t)=V^*U'(x,t)VU(x,t)^*$. 
Note that $W:X\times [0,1)\to U(M(A\otimes \K))$ is a norm continuous map satisfying $W(x_0,t)=1$ and $W(x,0)=1$. 
We claim that $W$ extends to a continuous map $\widetilde{W}:X\times I\to U(M(A\otimes \K))$ in the strict topology with  
$\widetilde{W}(x,1)=1$ for $x\in X$. 
Indeed, note that for $1/2\leq t<1$, we have $W(x,t)=V^*(\tilde{u}'(x,t)\otimes 1)V(\tilde{u}(x,t)^*\otimes 1)$, and 
the proof of Lemma \ref{lift} implies 
$$\lim_{t\to 1}\sup_{x\in X}\|\tilde{u}'(x,t)-p\tilde{u}(x,t)\|=0.$$
Thus we get 
$$\lim_{t\to 1}\sup_{x\in X}\|W(x,t)-V^*(\tilde{u}(x,t)\otimes 1)V(\tilde{u}(x,t)^*\otimes 1)\|=0,$$
and to verify the claim it suffices to show
$$\lim_{t\to 1}\sup_{x\in X}\|V^*(\tilde{u}(x,t)\otimes 1)V(\tilde{u}(x,t)^*\otimes 1)a-a\|=0,$$ 
$$\lim_{t\to 1}\sup_{x\in X}\|aV^*(\tilde{u}(x,t)\otimes 1)V(\tilde{u}(x,t)^*\otimes 1)-a\|=0,$$
for any $a\in A\otimes \K$. 
Note that Lemma \ref{lift},(2) implies  
$$\lim_{t\to 1}\sup_{x\in X}\|[\tilde{u}(x,t)\otimes 1,a]\|=0$$
for any $a\in A\otimes \K$. 
Since 
\begin{align*}
\lefteqn{\|V^*(\tilde{u}(x,t)\otimes 1)V(\tilde{u}(x,t)^*\otimes 1)a-a\|} \\
 & \leq \|[\tilde{u}(x,t)^*\otimes 1,a]\|
+\|V^*(\tilde{u}(x,t)\otimes 1)Va(\tilde{u}(x,t)^*\otimes 1)-a\|\\
 &= \|[\tilde{u}(x,t)^*\otimes 1,a]\|+\|[\tilde{u}(x,t)\otimes 1,Va]\|,
\end{align*}
\begin{align*}
\lefteqn{\|aV^*(\tilde{u}(x,t)\otimes 1)V(\tilde{u}(x,t)^*\otimes 1)-a\|} \\
 &\leq \|[aV^*,(\tilde{u}(x,t)\otimes 1)]\|+\|(\tilde{u}(x,t)\otimes 1)a(\tilde{u}(x,t)^*\otimes 1)-a\| \\
 &=\|[aV^*,(\tilde{u}(x,t)\otimes 1)]\|+\|[\tilde{u}(x,t)\otimes 1,a]\|,
\end{align*}
we get the claim. 

Thanks to the claim, we get $\Theta(\rho^{pu}_{x,t})=\Ad\widetilde{W}(x,t)\circ \rho^u_{x,t}$ 
for any $(x,t)\in X\times I$, and we may regard $\widetilde{W}$ as an element of $\Map(\Sigma X,U(M(A\otimes \K)))$, 
where $U(M(A\otimes \K))$ is equipped with the strict topology. 
Since $U(M(A\otimes \K))$ is contractible in the strict topology and the map $\Ad :U(M(A\otimes \K))\to \Aut(A\otimes \K)$ is continuous, we get the statement. 
\end{proof}

\section{Proof of Theorem \ref{main}}
Throughout this section, we assume that $A$ is a unital Kirchberg algebra and $(X,x_0)$ is a pointed compact metrizable 
space. 
Recall that $A$ is said to be in the Cuntz standard form if the $K_0$-class of $1_A$ is 0 in $K_0(A)$. 
Thanks to Lemma \ref{Morita}, we may and do assume that $A$ is in the Cuntz standard form in order to prove Theorem \ref{main}. 
We suppress $A$ (and sometimes even $X$) in $\Pi_{A,X}$ and $j_{A,X}$ to avoid heavy notation. 

For $u\in \Map(X,U(A_\flat))$, we choose $\tilde{u}$ and $U$ as in the proof of Lemma \ref{lift}. 
We often identify $u$ and $\tilde{u}$ if there is no possibility of confusion.  

Since we need an alternative description of the map $\overline{\chi}\circ \Pi_{X}$ discussed in Remark \ref{Dadarlat}, 
we work on asymptotic morphisms arising from $u\in \Map(X,U(A_\flat))$. 
The reader is referred to \cite[Section 24]{Bl} for the basics of asymptotic morphisms. 
For the parameter space of asymptotic morphisms, we adopt $[0,1)$ rather than the usual space $[1,\infty)$. 
Recall that two asymptotic morphisms $\phi$ and $\psi$ from $B$ to $D$ are asymptotically unitarily equivalent 
if there exists a continuous family $\{u_t\}_{t\in [0,1)}$ in $U(D)$ satisfying 
$$\lim_{t\to 1}\|u_t\phi_t(x)u_t^*-\psi_t(x)\|=0$$
for any $x\in B$.  

Let $\T=\{z\in \C;\; |z|=1\}.$
We denote by $\phi^u=(\phi^u_t)$ the ucp asymptotic morphism from 
$C(\T,A)$ to $C(X,A)$ corresponding to the homomorphism 
$$C(\T,A)\ni f\otimes a\mapsto f(u)a\in (C(X)\otimes A)^\flat.$$ 
Since $f(u(x_0))a=f(1)a$, the restriction of $\phi^u$ to $SA$ 
may be regarded as an asymptotic morphism from $SA$ to $C(X,x_0)\otimes A$.   
We denote by $KK(\phi^u|_{SA})$ the $KK$-class of $\phi^u$ restricted to $SA$, which is in  
$KK(SA,C(X,x_0)\otimes A)$.  

We denote by $b_{KK}$ the isomorphism from $KK(SA,B)$ to $KK(A,SB)$ given by the composition of two maps 
$$b_{KK}:KK(SA,B)\rightarrow KK(SSA,SB)\rightarrow KK(A,SB),$$ 
where the first one is given by suspension and the second one is given by the Bott periodicity.
We explicitly construct an asymptotic morphism giving the class $b_{KK}(KK(\phi^u|_{SA}))\in KK(A,SC(X,x_0)\otimes A)$. 

For $s\in [0,1]$ and $z\in \T$, we set 
$$W(s,z)=R(s)\left(
\begin{array}{cc}
\overline{z} &0  \\
0 &1 
\end{array}
\right)
R(s)^{-1}
\left(
\begin{array}{cc}
z &0  \\
0 &1 
\end{array}
\right),
$$
where
$$R(s)=\left(
\begin{array}{cc}
\sqrt{1-s} &-\sqrt{s}  \\
\sqrt{s} & \sqrt{1-s}
\end{array}
\right).
$$ 
Then the Bott projection $P(s,z)$ is given by 
$$P(s,z)=W(s,z)eW(s,z)^*,$$
where 
$$e=\left(
\begin{array}{cc}
1 &0  \\
0 &0 
\end{array}
\right).
$$
For $a\in A$, $t\in [0,1)$ and $(x,s)\in X\times I$, we set 
$$\psi^u_t(a)(x,s)=P(s,\tilde{u}(x,t))\left(
\begin{array}{cc}
a &0  \\
0 &a 
\end{array}
\right).$$ 
Then $\psi^u=(\psi^u_t)$ is an asymptotic morphism from $A$ to $M_2(C(\Sigma X)\otimes A)$, 
and we denote by $KK(\psi^u)\in KK(A,C(\Sigma X)\otimes A)$ the $KK$-class determined by $\psi^u$. 
Since 
$$\psi^u_t(a)(x,0)=\psi^u_t(a)(x,1)=\psi^u_t(a)(x_0,s)=\left(
\begin{array}{cc}
a &0  \\
0 &0 
\end{array}
\right),
$$
the difference $KK(\psi^u)-KK(j_{\Sigma X})$ falls in $KK(A,SC(X,x_0)\otimes A)$.

\begin{theorem}\label{KK} Let $A$ be a unital Kirchberg algebra, 
and let $(X,x_0)$ be a pointed compact metrizable space. 
Then we have 
$$b_{KK}(KK(\phi^u|_{SA}))=KK(\psi^u)-KK(j_{\Sigma X})=\overline{\chi}\circ \Pi_{X}([u]),$$ 
in $KK(A,SC(X,x_0)\otimes A)$.
\end{theorem}

\begin{proof}
It is routine work to show $b_{KK}(KK(\phi^u|_{SA}))=KK(\psi^u)-KK(j_{\Sigma X})$, and we concentrate on the proof of  
$KK(\psi^u)-KK(j_{\Sigma X})=\overline{\chi}\circ \Pi_{X}([u])$. 

Recall that $A$ is in the Cuntz standard form and it contains a unital copy of the Cuntz algebra $\cO_2$. 
Since $\tilde{u}(x,t)$ almost commutes with the unital copy of $\cO_2$ for $t$ sufficiently close to 1, the $K_1$-class 
of $\tilde{u}(\cdot,t)$ in $K_1(C(X)\otimes A)$ is trivial, and we may and do assume that $\tilde{u}(x,0)=1$. 
Thus we can define $\rho^{u,0}\in \Map(\Sigma X,\Aut(A))$ by 
$$\rho^{u,0}_{x,t}(a)=\left\{
\begin{array}{ll}
\Ad \tilde{u}(x,t) , &\quad  0\leq t<1 \\
\id , &\quad t=1
\end{array}
\right..
$$
For the definition of $\rho^u\in \Map(\Sigma X,\Aut(A\otimes \K))$, we can choose $U(x,t)$ 
of the form $\tilde{u}(x,t)\otimes 1$. 
With this choice, we have 
\begin{align*}
\overline{\chi}\circ \Pi_{X}([u])&=KK(\rho^u)-KK(j_{A\otimes \K,\Sigma X}) \\
 &=KK(\rho^{u,0}\otimes\id_\K)-KK(j_{A,\Sigma X}\otimes\id_\K)\\
 &=KK(\rho^{u,0})-KK(j_{A,\Sigma X}). 
\end{align*}
Thus our task is to show $KK(\psi^u)=KK(\rho^{u,0})$. 

We fix a homeomorphism $f:[0,1)\to [0,\infty)$, e.g. $f(t)=t/(1-t)$, and set 
$v(x,y)=\tilde{u}(x,f^{-1}(y))$ for $(x,y)\in X\times [0,\infty)$.  
Let 
$$V_t(x,s)=\left(
\begin{array}{cc}
v(x,sf(t)) &0  \\
0 &v(x,sf(t))^* 
\end{array}
\right)W(s,\tilde{u}(x,t))^*,
$$
for $t\in [0,1)$ and $(x,s)\in X\times I$. 
Then we have $V_t(x,0)=V_t(x,1)=V(x_0,s)=1$,
and $(V_t)_{t\in [0,1)}$ is a continuous family of unitaries in $M_2(C(\Sigma X)\otimes A)$. 
For $a\in A$, we have 
\begin{eqnarray*}\lefteqn{
V_t(x,s)\psi^u_t(a)V_t(x,s)^*}\\
&=&\Ad
\left(
\begin{array}{cc}
v(x,sf(t)) &0  \\
0 &v(x,sf(t))^* 
\end{array}
\right)
\Big(e
W(s,\tilde{u}(x,t))^*
\left(
\begin{array}{cc}
a &0  \\
0 &a 
\end{array}
\right)
W(s,\tilde{u}(x,t))
\Big).
\end{eqnarray*}
Since 
$$\lim_{t\to1}\sup_{(x,s)\in X\times [0,1]}\|[e,
W(s,\tilde{u}(x,t))^*
\left(
\begin{array}{cc}
a &0  \\
0 &a 
\end{array}
\right)
W(s,\tilde{u}(x,t))]\|=0,$$
the asymptotic morphism $A\ni a\mapsto V_t\psi^u_t(a)V_t^*$ is equivalent to 
\begin{eqnarray*}
\lefteqn{a\mapsto 
\Ad\left(
\begin{array}{cc}
v(x,sf(t)) &0  \\
0 &v(x,sf(t))^* 
\end{array}
\right)
\Big(e
W(s,\tilde{u}(x,t))^*
\left(
\begin{array}{cc}
a &0  \\
0 &a 
\end{array}
\right)
W(s,\tilde{u}(x,t))
e\Big) 
} \\
 &=&\Ad \left(
\begin{array}{cc}
v(x,sf(t)) &0  \\
0 &v(x,sf(t))^* 
\end{array}
\right) \\
&&
\Big(
\left(
\begin{array}{cc}
\tilde{u}(x,t)^* &0  \\
0 &0 
\end{array}
\right)
R(s)
\left(
\begin{array}{cc}
\tilde{u}(x,t)a\tilde{u}(x,t)^* &0  \\
0 &a 
\end{array}
\right)
R(s)^*
\left(
\begin{array}{cc}
\tilde{u}(x,t) &0  \\
0 &0 
\end{array}
\right)
\Big)\\
 &=&\Ad \left(
\begin{array}{cc}
v(x,sf(t)) &0  \\
0 &v(x,sf(t))^* 
\end{array}
\right) 
\Big(
\left(
\begin{array}{cc}
(1-s)a+s\tilde{u}(x,t)^*a\tilde{u}(x,t) &0  \\
0 &0 
\end{array}
\right)
\Big).
\end{eqnarray*}
Thus the $KK$-class of $\psi$ is the same as that of the asymptotic morphism $\mu^u_t:A\rightarrow C(\Sigma X)\otimes A$ given by 
$$\mu^u_t(a)(x,s)=\Ad v(x,sf(t))\big((1-s)a+s\tilde{u}(x,t)^*a\tilde{u}(x,t)\big).$$

We construct a homotopy between $\mu^u$ and $\rho^{u,0}$ now. 
For $a\in A$, $r\in [0,1]$, $x\in X$, $s\in [0,1)$, and $t\in [0,1)$, we set 
\begin{eqnarray*}\lefteqn{\Psi_t(a)(r,x,s)}\\
&=&\Ad v(x,(1-r)sf(t)+rf(s))\\
&&\big((1-s+rs)a+(s-rs)\Ad v(x,(1-r)f(t)+rf(s))^*(a)\big).
\end{eqnarray*}
Then for fixed $(r,x,s,t)\in I\times X\times [0,1)^2$ the map $A\ni a\mapsto \Psi_t(a)(r,x,s)$ is a ucp map, and 
for a fixed $a\in A$ the map $I\times X\times [0,1)^2\ni (r,x,s,t)\in \Psi_t(a)(r,x,s)\in A$ is continuous. 
We also have $\Psi_t(a)(r,x,0)=\Psi_t(a)(r,x_0,s)=a$. 

We first show  
$$\lim_{s\to 1}\sup_{r\in [0,1],\;x\in X}\|\Psi_t(a)(r,x,s)-a\|=0.$$
We have \begin{eqnarray*}
\lefteqn{\|\Psi_t(a)(r,x,s)-a\|} \\
 &\leq&(1-s+rs)\|[v(x,(1-r)sf(t)+rf(s)),a]\|\\
 &+&(s-rs)\|[v(x,(1-r)sf(t)+rf(s))v(x,(1-r)f(t)+rf(s))^*,a]\|.
\end{eqnarray*}
Let $0<\epsilon\leq 1$ be given. 
Since 
$$\lim_{y\to\infty}\sup_{x\in X}\|[v(x,y),a]\|=0,$$ 
there exits  $y_0\geq 0$ so that $\|[v(x,y),a]\|\leq \epsilon$ holds for every $y\geq y_0$. 
We choose $0\leq s_0<1$ satisfying $y_0\leq \epsilon f(s_0)$.  
Since $v(x,y)$ is uniformly continuous on $X\times [0,y_0+2]$, there exists $0<\delta<1$ such that 
$\|v(x,y_1)-v(x,y_2)\|<\epsilon$ holds for any $x\in X$ and $y_1,y_2\in [0,y_0+2]$ with $|y_1-y_2|<\delta$.   
Assume $\max\{s_0,1-\epsilon,1-\frac{\delta}{f(t)}\}< s<1$. 
Then for $0\leq r\leq \epsilon$, we have 
$$(1-s+rs)\|[v(x,(1-r)sf(t)+rf(s)),a]\|\leq 2(1-s+\epsilon)\|a\|\leq 4\epsilon\|a\|,$$ 
$$(s-rs)\|[v(x,(1-r)sf(t)+rf(s))v(x,(1-r)f(t)+rf(s))^*,a]\|\leq 2\epsilon\max\{1,\|a\|\}.$$
The first estimate is straightforward, and the second estimate is given as follows. 
Let $y_1=(1-r)sf(t)+rf(s)$, $y_2=(1-r)f(t)+rf(s)$. 
Then $y_1<y_2$ and $y_2-y_1=(1-r)f(t)(1-s)\leq \delta$. 
If $y_0\leq (1-r)sf(t)+rf(s)$, we get $\|[v(x,y_1),a]\|\leq \epsilon$ and $\|[v(x,y_2),a]\|\leq \epsilon$. 
If $(1-r)sf(t)+rf(s)<y_0$, we get $\|v(x,y_1)v(x,y_2)^*-1\|<\epsilon$, and we get the second estimate.
For $\epsilon\leq r\leq 1$, since we have $y_0<rf(s)$, we get
$$(1-s+rs)\|[v(x,(1-r)sf(t)+rf(s)),a]\|<\epsilon,$$
$$(s-rs)\|[v(x,(1-r)sf(t)+rf(s))v(x,(1-r)f(t)+rf(s))^*,a]\|<2\epsilon.$$
Thus we get 
$$\|\Psi_t(a)(r,x,s)-a\|\leq \epsilon\max\{3,6\|a\|,4\|a\|+2\},$$
and the claim is shown. 
Thus we may and do regard $\Psi_t$ as a ucp map from $A$ to 
$C([0,1]\times \Sigma X)\otimes A$ satisfying $\Psi_t(a)(r,\Sigma x_0)=a$. 

Next we show that $t\mapsto \Psi_t(a)\in C([0,1]\times \Sigma X)\otimes A$ is continuous. 
Let $t_1\in [0,1)$. 
For $a\in A$ and $\epsilon>0$, we choose $y_0$ as before. 
Since $v(x,y)$ is uniformly continuous on $X\times [0,y_0+f(t_1)+1]$, there exists $0<\delta_2<1$ satisfying 
$\|v(x,y)-v(x,z)\|\leq \epsilon$ for any $(x,y), (x,z)\in X\times [0,y_0+f(t_1)+1]$ with $|y-z|<\delta_2$. 
Since $f$ is continuous, we choose $0<\delta_3$ satisfying $|f(t)-f(t_1)|<\delta_2$ for any $t\in [0,1)$ with $|t-t_1|<\delta_3$. 
Then for any $(r,x,s)\in I\times X\times [0,1)$ and $t_2\in [0,1)$ with $|t_2-t_1|<\delta_3$, we have
\begin{align*}
\lefteqn{\|\Psi_{t_1}(a)(r,x,s)-\Psi_{t_2}(a)(r,x,s)\| } \\
&\leq \| \Ad v(x,y_1)(a)-\Ad v(x,y_2)(a) \|\\
 &+\|\Ad v(x,y_1)v(x,z_1)^*(a)-\Ad v(x,y_2)v(x,z_2)^*(a)\|\\
 &=\|[v(x,y_2)^*v(x,y_1),a]\|+\|[v(x,z_2)v(x,y_2)^*v(x,y_1)v(x,z_1)^*,a]\|, 
\end{align*}
where $y_i=(1-r)sf(t_i)+rf(s)$ and $z_i=(1-r)f(t_i)+rf(s)$.
If $y_0\leq rf(s)$, we get 
$$\|\Psi_{t_1}(a)(r,x,s)-\Psi_{t_2}(a)(r,x,s)\|\leq 6\epsilon.$$ 
Assume $rf(s)<y_0$. 
Then we have $y_1,y_2,z_1,z_2\in [0,y_0+f(t_1)+1]$, and $|y_1-y_2|<\delta_2$, $|z_1-z_2|<\delta_2$. 
Thus we get 
$$\|\Psi_{t_1}(a)(r,x,s)-\Psi_{t_2}(a)(r,x,s)\|\leq 6\epsilon \|a\|,$$
and so the map $t\mapsto \Psi_t(a)$ is continuous. 

Finally we show that $\Psi_t$ is asymptotically multiplicative. 
Let $a_1,a_2\in A$, and let $\epsilon>0$ be given. 
We choose $z_0\geq 0$ such that $\|[v(x,y),a_i]\|<\epsilon$ holds for any $x\in X$, $y\geq z_0$, and $i=1,2$. 
We choose $t_0\in (0,1)$ satisfying $z_0<\epsilon f(t_0)$. 
Let $t_0<t<1$. 
Then we have 
\begin{align*}
\lefteqn{\|\Psi_t(a_1a_2)(r,x,s)-\Psi_t(a_1)(r,x,s)\Psi_t(a_2)(r,x,s)\|} \\
 &=\|\big((1-s+rs)a_1a_2+(s-rs)\Ad v(x,(1-r)f(t)+rf(s))^*(a_1a_2)\big)\\
&-\big((1-s+rs)a_1+(s-rs)\Ad v(x,(1-r)f(t)+rf(s))^*(a_1)\big)\\
&\times\big((1-s+rs)a_2+(s-rs)\Ad v(x,(1-r)f(t)+rf(s))^*(a_2)\big)\|.
\end{align*}
Setting $b_i=\Ad v(x,(1-r)f(t)+rf(s))^*(a_i)-a_i$ for $i=1,2$, we get 
\begin{align*}
\lefteqn{\|\Psi_t(a_1a_2)(r,x,s)-\Psi_t(a_1)(r,x,s)\Psi_t(a_2)(r,x,s)\|} \\
&=\|(1-s+rs)a_1a_2+s(1-r)(a_1+b_1)(a_2+b_2)
-(a_1+s(1-r)b_1)(a_2+s(1-r)b_2)\|\\
&=s(1-r)(1-s(1-r))\|b_1b_2\|. 
\end{align*}
If $1-\epsilon\leq  r\leq 1$, we have 
$$\|\Psi_t(a_1a_2)(r,x,s)-\Psi_t(a_1)(r,x,s)\Psi_t(a_2)(r,x,s)\|\leq 4\|a_1\|\|a_2\|\epsilon.$$
If $0\leq r\leq 1-\epsilon$, we have $z_0<(1-r)f(t)+rf(s)$, and $\|b_i\|\leq \epsilon$ for $i=1,2$, 
and so   
$$\|\Psi_t(a_1a_2)(r,x,s)-\Psi_t(a_1)(r,x,s)\Psi_t(a_2)(r,x,s)\|\leq \epsilon^2.$$
This implies 
$$\lim_{t\to 1}\sup_{(r,x,s)\in [0,1]\times \Sigma X}\|\Psi_t(a_1a_2)(r,x,s)-\Psi_t(a_1)(r,x,s)\Psi_t(a_2)(r,x,s)\|=0,$$
and hence $\Psi_t$ is asymptotically multiplicative.   

Now we know that $\Psi_t$ is an asymptotic morphism from $A$ to $C([0,1]\times \Sigma X)\otimes A$. 
Since $\Psi_t(a)(0,x,s)=\mu^u_t(a)(x,s)$ and $\Psi_t(a)(1,x,s)=\rho^{u,0}_{x,s}(a)$, we get the desired homotopy. 
\end{proof}

\begin{remark}\label{generalization} Let $Y$ be another compact metrizable space. 
Then we can generalize our arguments so far to the following setting. 
Let $B=C(Y)\otimes A$, and let $\Aut_{C(Y)}(B\otimes \K)$ be the set of $C(Y)$-linear automorphisms of $B\otimes \K$. 
Then we can construct a group homomorphism 
$$\Pi^Y_{A,X}:[X,U(B^\flat\cap A')]\to [X,\Omega\Aut_{C(Y)}(B\otimes \K)]$$ that is natural in $X$. 
Now the Dadarlat map $\overline{\chi}$ is generalized to 
$$\overline{\chi}^Y:[X,\Omega\Aut_{C(Y)}(B\otimes \K)]\to KK(A,SC(X,x_0)\otimes B),$$ 
which is a group homomorphism. 
On the other hand, we can construct an asymptotic morphism $\psi^{u,Y}$ for $u\in \Map(X,U(B^\flat\cap A'))$ 
from $A$ to $M_2(C(\Sigma X)\otimes B)$ in the same way, and we get 
$$KK(\psi^{u,Y})-KK(j_{B,\Sigma X})=\overline{\chi}^Y\circ \Pi^Y_{A,X}([u])\in KK(A,SC(X,x_0)\otimes C(Y)\otimes A).$$ 
\end{remark}

Since $[X,U(A_\flat)]$ is identified with $K_1(C(X,x_0)\otimes A_\flat)$ thanks to Lemma \ref{identification}, the Bott periodicity gives rise to 
an isomorphism 
$$\beta:[X,U(A_\flat)]\rightarrow [\Sigma^2X,U(A_\flat)].$$
On the other hand, we denote by $\beta'$ the isomorphism 
$$\beta':KK(A,SC(X,x_0)\otimes A)\rightarrow KK(A,SC(\Sigma^2X,\Sigma^2x_0)\otimes A).$$
given by the the Bott periodicity in the $KK$-side. 

\begin{lemma}\label{Bott} 
Let $A$ be a unital Kirchberg algebra, and let $(X,x_0)$ be a pointed compact metrizable space. 
Then the following diagram is commutative. 
$$\begin{CD}
[X,U(A_\flat)]@>\beta>> [\Sigma^2X,U(A_\flat)]\\
@V{\overline{\chi}\circ \Pi_{X}}VV @VV{\overline{\chi}\circ\Pi_{\Sigma^2 X}} V\\
KK(A,SC(X,x_0)\otimes A)@>\beta'>>KK(A,SC(\Sigma^2X,\Sigma^2x_0)\otimes A)
\end{CD}
$$
\end{lemma}

\begin{proof}  
We identify $\Sigma^2X$ with $(X\times I^2)/((\{x_0\}\times I^2)\cup (X\times \partial I^2))$, 
and in particular $S^2$ with $I^2/\partial I^2$. 
We denote by $\{e^{(k)}_{ij}\}$ the canonical system of matrix units of the matrix algebra $M_k(\C)$. 
We set $e_B(s_1,s_2)=P(s_1,e^{2\pi \sqrt{-1}s_2})$ for $(s_1,s_2)\in I^2$, which is the Bott projection. 
We choose two isometries $S_1,S_2\in A_\flat$ with mutually orthogonal ranges.  
We denote by $\Phi$ the homomorphism from $M_2(\C)$ to $A_\flat$ sending $e^{(2)}_{ij}$ to $S_iS_j^*$. 
We set $q=S_1S_1^*+S_2S_2^*$. 

For a given $u\in \Map(X,U(A_\flat))$, we define $u_1,u_2\in U(C(X\times I^2,A_\flat))$ by 
$$u_1(x,s_1,s_2)=\Phi(e_B(s_1,s_2))(S_1u(x)S_1^*+S_2u(x)S_2^*)+\Phi(1-e_B(s_1,s_2))+1-q,$$
\begin{eqnarray*}
u_2(x,s_1,s_2)&=&\Phi(e^{(2)}_{11})(S_1u(x)S_1^*+S_2u(x)S_2^*)+\Phi(1-e^{(2)}_{11})+1-q\\
&=&S_1u(x)S_1^*+S_2S_2^*+1-q. 
\end{eqnarray*}
Then $u_1(x_0,s_1,s_2)=u_2(x_0,s_1,s_2)=1$, and for $(s_1,s_2)\in \partial I^2$ we have 
$$u_1(x,s_1,s_2)=u_2(x,s_1,s_2)=S_1u(x)S_1^*+S_2S_2^*+1-q.$$
Thus the map $u':=u_1u_2^*$ belongs to $\Map(\Sigma^2X,U(A_\flat))$, and we have $\beta([u])=[u'].$
Theorem \ref{KK} implies 
$$\overline{\chi}\circ \Pi_{A,X}([u])=KK(\psi^{u})-KK(j_{\Sigma X})\in KK(A,SC(X,x_0)\otimes A),$$ 
$$\overline{\chi}\circ\Pi_{A,\Sigma^2 X}([u'])=KK(\psi^{u'})-KK(j_{\Sigma^3 X})\in KK(A,SC(\Sigma^2X,\Sigma^2 x_0)\otimes A).$$ 

Through the split exact sequence 
$$0\to C_0(I^2\setminus \partial I^2)\rightarrow C(S^2)\rightarrow \C\rightarrow 0,$$
we regard $KK(A,SC(\Sigma^2X,\Sigma^2x_0)\otimes A)$ as a subgroup of 
$KK(A,SC(X,x_0)\otimes B)$, where $B=C(S^2)\otimes A$. 
Since 
$$u_1,u_2\in \Map(X,U(C(S^2)\otimes A_\flat))\subset \Map(X,U(B^\flat\cap A')),$$ 
we can apply the argument in Remark \ref{generalization} with $Y=S^2$. 
As there exists an obvious homomorphism 
$$[\Sigma^2X,\Omega\Aut(A\otimes \K)]\to [X,\Omega\Aut_{C(S^2)}(B\otimes \K)],$$ 
and the diagram 
$$
\begin{CD}
[\Sigma^2X,\Omega\Aut(A\otimes \K)]@>>> [X,\Omega\Aut_{C(S^2)}(B\otimes \K)]\\
@V{\overline{\chi}}VV @VV{\overline{\chi}^{S^2}}V\\
KK(A,SC(\Sigma^2X,\Sigma^2x_0)\otimes A)@>\mathrm{inclusion}>>KK(A,SC(X,x_0)\otimes B)
\end{CD}$$
is commutative, we get 
\begin{align*}
\overline{\chi}\circ \Pi_{\Sigma^2X}([u'])&=\overline{\chi}^{S^2}\circ \Pi^{S^2}_{A,X}([u'])  \\
 &=\overline{\chi}^{S^2}\circ \Pi^{S^2}_{A,X}([u_1])-\overline{\chi}^{S^2}\circ \Pi^{S^2}_{A,X}([u_2])\\
 &=KK(\psi^{u_1,S^2})-KK(\psi^{u_2,S^2}). 
\end{align*}
Recall that the two asymptotic morphisms $\psi^{u_1,S^2}$ and $\psi^{u_1,S^2}$ arise from the homomorphisms 
from $A$ to $M_2(C(\Sigma X\times S^2)\otimes A^\flat)$ given by the left multiplication of projections 
$P(r,u_1(x,s_1,s_2))$ and  $P(r,u_2(x,s_1,s_2))$ respectively.

Direct computation yields 
\begin{align*}
\lefteqn{P(r,u_1(x,s_1,s_2))} \\
 &=\big((1\otimes S_1)P(r,u(x))(1\otimes S_1^*)+(1\otimes S_2)P(r,u(x))(1\otimes S_2^*)\big)(1\otimes \Phi(e_B(s_1,s_2))) \\
 &+e^{(2)}_{11}\otimes (\Phi(1-e_B(s_1,s_2))+1-q), 
\end{align*}
$$P(r,u_2(x,s_1,s_2))=(1\otimes S_1)P(r,u(x))(1\otimes S_1^*)+e^{(2)}_{11}\otimes (S_2S_2^*+1-q).$$
Note that the two projections $e^{(3)}_{11}\otimes 1\otimes q$ and $(e^{(3)}_{11}+e^{(3)}_{22})\otimes 1\otimes 1$ 
are equivalent in $M_3(\C)\otimes \C1_{M_2(\C)}\otimes A_\flat$, and their complements are also equivalent 
as they are full properly infinite projections with the same $K_0$-class. 
Thus there exists a partial isometry $V\in M_3(\C)\otimes \C1_{M_2(\C)}\otimes A_\flat$ satisfying 
$V^*V=1-e^{(3)}_{11}\otimes 1\otimes q$ and $VV^*=e^{(3)}_{33}\otimes 1\otimes 1$. 
We set  
$$W=e^{(3)}_{11}\otimes 1\otimes S_1^*+e^{(3)}_{21}\otimes 1\otimes S_2^*+V,$$
which is a unitary in $M_3(\C)\otimes \C 1_{M_2}\otimes A_\flat$. 
Thus the two asymptotic morphisms 
$$A\ni a\mapsto e_{11}^{(3)}\otimes \psi_t^{u_i,S^2}(a)\in M_3(\C)\otimes M_2(C(\Sigma X\times S^2)\otimes A),\quad i=1,2$$
are asymptotically unitarily equivalent to the ones corresponding to the homomorphisms from $A$ to 
$M_3(\C)\otimes M_2(C(\Sigma X\times S^2)\otimes A^\flat)$ given by the left multiplication of the projections 
$$W(e^{(3)}_{11}\otimes P(r,u_i(x,s_1,s_2)))W^*,\quad i=1,2,$$
respectively.  

Let $\varphi:M_2(\C)\to M_3(\C)$ be the embedding map given by $\varphi(e^{(2)}_{ij})=e^{(3)}_{ij}$.  
Then 
\begin{align*}
\lefteqn{W(e^{(3)}_{11}\otimes P(r,u_1(x,s_1,s_2)))W^*} \\
 &=\varphi(e_B(s_1,s_2))\otimes P(r,u(x))+\varphi(1-e_B(s_1,s_2))\otimes e^{(2)}_{11}\otimes 1\\
 &+V(e^{(3)}_{11}\otimes e^{(2)}_{11}\otimes (1-q))V^*,
\end{align*}
\begin{align*}
\lefteqn{W(e^{(3)}_{11}\otimes P(r,u_2(x,s_1,s_2)))W^*} \\
 &=e^{(3)}_{11}\otimes P(r,u(x))+e^{(3)}_{22}\otimes e^{(2)}_{11}\otimes 1 
 +V(e^{(3)}_{11}\otimes e^{(2)}_{11}\otimes (1-q))V^*.
\end{align*}
Now we can see 
\begin{align*}
\lefteqn{KK(\psi^{u_1,S^2})-KK(\psi^{u_2,S^2})} \\
 &=KK(e_B\otimes \psi^u)+KK((1_{M_2(C(S^2))}-e_B)\otimes j_{\Sigma X})\\
 &-KK(1_{C(S^2)}\otimes \psi^u)-KK(1_{C(S^2)}\otimes j_{\Sigma X}) \\
 &=KK(e_B\otimes \psi^u)-KK(1_{C(S^2)}\otimes \psi^u)+KK(1_{C(S^2)}\otimes j_{\Sigma X})-KK(e_B\otimes j_{\Sigma X})\\
 &=\beta'(KK(\psi^u)-KK(j_{\Sigma X}))\\
 &=\beta'\circ \overline{\chi}\circ \Pi_{X}([u]),
\end{align*}
which finishes the proof.  
\end{proof}

Thanks to Lemma \ref{Bott}, it suffices to show the statement of Theorem \ref{main} for $n=0$ and $n=1$ in order to 
prove the general statement. 
Recall that Dadarlat showed that $\overline{\chi}:\pi_1(\Aut(A\otimes \K))\to KK(A,SA)$ is an isomorphism. 

\begin{lemma}\label{K_12} Let $A$ be a unital Kirchberg algebra. 
Then $\Pi_0$ is an isomorphism from $\pi_0(U(A_\flat))=K_1(A_\flat)$ to $\pi_0(\Omega\Aut(A\otimes \K))=\pi_1(\Aut(A\otimes \K))$.
\end{lemma}

\begin{proof} 
Since $S^0=\{1,-1\}$ with base point 1, a map $f\in \Map(S^0,Y)$ is identified with $f(-1)\in Y$, and 
we identify $\Map(S^0,Y)$ with $Y$. 
Recall that we may and do assume that $A$ is in the Cuntz standard form. 

The surjectivity of $\Pi_0$ was essentially proved in \cite[Proposition 8.4]{IM2010}. 
Namely, for any $x\in KK(A,SA)$, there exists a $\Z^2$-action $\alpha$ on $A$ such that 
$KK(\alpha_g)=KK(\id)$ for any $g\in \Z^2$, and $x$ is realized as an invariant $\Phi(\alpha)=x$ defined in \cite[p.393]{IM2010}. 
We first recall the definition of $\Phi(\alpha)$. 
Since $\alpha_{(1,0)}\otimes \id_\K$ is homotopic to $\id$ in $\Aut(A\otimes \K)$, we choose a continuous path 
$\{\gamma_t\}_{t\in [0,1]}$ from $\id$ to $\alpha_{(1,0)}\otimes \id_{\K}$. 
Then 
$$\sigma_t=\gamma_t^{-1}\circ (\alpha_{(0,1)}\otimes \id_\K)\circ \gamma_t\circ (\alpha_{(0,1)}^{-1}\otimes \id_\K)$$ 
is a loop in $\Aut(A\otimes \K)$ satisfying $\sigma_0=\id$, and $\Phi(\alpha)$ in \cite{IM2010} is defined by $\overline{\chi}([\sigma])$, 
which does not depend on the choice of $\gamma$. 

Since $KK(\alpha_{(1,0)})=KK(\id)$, there exists a continuous family of unitaries $\{v(t)\}_{t\in [0,1)}$ in $A$ satisfying 
$\lim_{t\to 1}\Ad v(t)=\alpha_{(1,0)}$. 
As before we can choose a norm continuous family of unitaries $\{V(t)\}_{t\in [0,1)}$ in $M(A\otimes \K)$ satisfying 
$V(t)=v(t)\otimes 1$ for $1/2\leq t<1$ and $V(0)=1$. 
Now we can adopt $\Ad V(t)$ for $\gamma_t$, and with this choice 
$$\sigma_t=\Ad (V(t)^*(\alpha_{(0,1)}\otimes \id_\K)(V(t))),$$
for $t\in [0,1)$. 
Note that for $1/2<t<1$, we have 
$$V(t)^*(\alpha_{(0,1)}\otimes \id_\K)(V(t))=v(t)^*\alpha_{(0,1)}(v(t))\otimes 1.$$ 
Thus setting $u(t)=v(t)^*\alpha_{(0,1)}(v(t))$ for $1/2\leq t<1$, and $u(t)=u(1/2)$ for $0\leq t<1/2$, 
we get $u\in U(A_\flat)$ satisfying $\overline{\chi}(\Pi_{0}([u]))=x$. 

We prove the injectivity of $\Pi_{0}$ now. 
Assume that $u\in U(A_\flat)$ and $[u]$ is in the kernel of $\Pi_{0}$. 
As in the proof of Theorem \ref{KK}, we may and do assume $u(0)=1$. 
Since 
$$0=\overline{\chi}\circ \Pi_{X}([u])=KK(\psi^u)-KK(j),$$
we get $KK(\psi^u)=KK(j)$. 
By the Kirchberg-Phillips classification theorem \cite{P00}, the asymptotic morphism $(\psi^u_t)$ and $j\oplus 0$ are 
asymptotically unitarily equivalent, and there exists a unitary path $\{V_t\}_{t\in [0,1)}$ in 
$M_2(C(S^1)\otimes A)$ satisfying 
\begin{equation}\label{A1}
\lim_{t\to1}\sup_{s\in S^1}\|V_t(s)P(s,u(t))\left(
\begin{array}{cc}
a &0  \\
0 &a 
\end{array}
\right)V_t(s)^*-\left(
\begin{array}{cc}
a &0  \\
0 &0 
\end{array}
\right)\|
=0,
\end{equation}
for any $a\in A$. 
In particular, we have the following by setting $a=1$:
$$\lim_{t\to1}\sup_{s\in [0,1]}\|[V_t(s)R(s)\left(
\begin{array}{cc}
u(t)^* &0  \\
0 &1 
\end{array}
\right)R(s)^{-1}
,e]\|=0.$$
Let 
$$D(s,t)=V_t(s)R(s)\left(
\begin{array}{cc}
u(t)^* &0  \\
0 &1 
\end{array}
\right)R(s)^{-1}.$$
Then $D$ is asymptotically diagonal, that is, there exist $d_1,d_2\in C(I\times [0,1),U(A))$ satisfying 
$$\lim_{t\to1}\sup_{s\in I}\|D(s,t)-\left(
\begin{array}{cc}
d_1(s,t) &0  \\
0 &d_2(s,t) 
\end{array}
\right)\|=0.$$
Since $V_t(1)=V_t(0)$ and 
$$D(0,t)=V_t(0)\left(
\begin{array}{cc}
u(t)^* &0  \\
0 &1 
\end{array}
\right),
$$we get 
$$D(1,t)=V_t(1)\left(
\begin{array}{cc}
1 &0  \\
0 &u(t)^* 
\end{array}
\right)=D(0,t)\left(
\begin{array}{cc}
u(t) &0  \\
0 &u(t)^* 
\end{array}
\right),
$$
and 
$$\lim_{t\to 1}\|d_1(1,t)-d_1(0,t)u(t)\|=0.$$ 
(\ref{A1}) implies 
$$\lim_{t\to1}\sup_{s\in I}\|
D(s,t)eR(s)
\left(
\begin{array}{cc}
u(t)au(t)^* &0  \\
0 &a 
\end{array}
\right)
R(s)^*D(s,t)^*
-\left(
\begin{array}{cc}
a &0  \\
0&0 
\end{array}
\right)
\|=0.$$
Since 
$$\lim_{t\to 1}\|u(t)au(t)^*-a\|=0,$$
we may replace $u(t)au(t)^*$ with $a$ in the above formula, and we get 
$$\lim_{t\to1}\sup_{s\in I}\|[d_1(s,t),a]\|=0,$$
for all $a\in A$. 
Let $v(s,t)=d_1(0,t)^{-1}d_1(s,t)$. 
Then $v$ is continuous on $I\times [0,1)$, $v(0,t)=1$, and $v(1,\cdot)-u(\cdot)\in C_0 A$. 
Moreover 
$$\lim_{t\to1}\sup_{s\in [0,1]}\|[v(s,t),a]\|=0,$$ 
for all $a\in A$. 
Thus Theorem \ref{refinedNakamura}, which is a refinement of Nakamura's homotopy theorem, implies that $u$ is homotopic to 1 in $U(A_\flat)$. 
\end{proof}

Instead of proving the statement of Theorem \ref{main} for $n=1$, we directly construct an isomorphism 
from $K_0(A_\flat)$ to $KK(A,A)$ by using asymptotic morphisms. 

\begin{lemma}\label{K0} Let $A$ be a unital Kirchberg algebra. 
For a non-zero projection $p\in A_\flat$, we denote by $\phi^p=(\phi^p_t)_{t\in [0,1)}$ 
the asymptotic morphism from $A$ to itself arising from the homomorphism $A\ni a \mapsto pa\in A^\flat$. 
Let $KK(\phi^p)$ be the $KK$-class for $\phi^p$. 
Then the map 
$$K_0(A_\flat)\ni K_0(p)\mapsto KK(\phi^p)\in KK(A,A),$$ 
is a group isomorphism.
\end{lemma}

\begin{proof} We may assume that $A$ is in the Cuntz standard form as before. 
Let $p$ be a non-zero projection in $A_\flat$. 
Since $p$ commutes with a unital copy of $\cO_2$ in $A^\flat$, 
the $K_0$-class of $p$ in $A^\flat$ is trivial. 
Since $1$ and $p$ are properly infinite projections in the same class in $K_0(A^\flat)$, 
there exists an isometry $v\in A^\flat$ whose range projection is $p$. 
We may assume that $v(t)$ is an isometry for every $t\in [0,\infty)$. 
We set $\psi_t^v(a)=v(t)^*av(t)$ for $a\in A$. 
Then $\psi^v=(\psi^v_t)_{t\in [0,1)}$ is a ucp asymptotic morphism from $A$ to $A$. 
Since $\phi^p\oplus 0$ and $\psi^v\oplus 0$ are asymptotically unitarily equivalent as asymptotic morphisms from 
$A$ to $M_2(A)$, they give the same $KK$-class. 

Assume that $KK(\phi^p)=KK(\phi^q)$ for non-zero projections $p,q\in A_\flat$. 
Let $v,w$ be isometries in $A^\flat$ satisfying $vv^*=p$, $ww^*=q$.  
Then there exists a unitary $u\in C^bA$  satisfying 
$\lim_{t\to 1}\|u(t)\psi^v_t(a)u(t)^*-\psi^w_t(a)\|=0,\quad \forall a\in A.$
This means that $wuv^*\in A_\flat$, and $[p]=[q]$ in $K_0(A_\flat)$. 
Thus the map $K_0(p)\mapsto  KK(\phi^p)$ is injective. 

Assume $x\in KK(A,A)$ now. 
Since $A$ is in the Cuntz standard form, there exist unital endomorphisms 
$\rho_1,\rho_2$ of $A$ satisfying 
$$KK(\rho_1)=KK(\id)-KK(\rho_2)=x.$$
Taking isometries $S_1,S_2\in A$ satisfying the $\cO_2$ relation, we set 
$$\sigma(a)=S_1\rho_1(a)S_1^*+S_2\rho_2(a)S_2^*,\quad a\in A.$$
Since $KK(\sigma)=KK(\id)$, there exists a unitary $z\in C^bA$ satisfying 
$$\lim_{t\to1}\|\sigma(a)-z(t)az(t)^*\|=0,\quad \forall a\in A,$$
and so 
$$\lim_{t\to 1}\|\rho_1(a)-S_1^*z(t)az(t)^*S_1\|=0,\quad \forall a\in A.$$
This implies that $p=z^*S_1S_1^*z$ is a projection in $A_\flat$, and $KK(\phi^p)=x$. 
\end{proof}

\begin{remark}
The $KK$-group $KK(A,A)$ is a ring with the Kasparov product, 
and this product is given by composition of asymptotic morphisms. 
By the isomorphism $KK(A,A)\cong K_0(A_\flat)$ stated in the lemma above, 
we can transfer the product to $K_0(A_\flat)$, 
which is described as follows. 
Let $p,q\in A_\flat$ be projections and 
let $\phi^p,\phi^q$ be the asymptotic morphisms discussed above. 
We can find a homeomorphism $r_0:[0,1)\to[0,1)$ such that 
for any homeomorphism $r:[0,1)\to[0,1)$ with $r(t)\geq r_0(t)$ 
one has $\lim_t\lVert[p(r(t)),q(t)]\rVert=0$. 
Then, $(p(r(t))q(t))_{t\in[0,1)}$ gives rises to a projection in $A_\flat$, 
and the map $t\mapsto\phi^p_{r(t)}\circ\phi^q_t$ is an asymptotic morphism. 
By \cite{CH}, 
this asymptotic morphism is the composition of $\phi^p$ and $\phi^q$. 
Thus, the product of $K_0(p)$ and $K_0(q)$ in $K_0(A_\flat)$ is equal to 
the $K_0$-class of the projection $(p(r(t))q(t))_{t\in[0,1)}$. 
\end{remark}

\begin{proof}[Proof of Theorem \ref{main}] 
Since $\overline{\chi}:\pi_n(\Omega\Aut(A\otimes \K))\to KK(A,S^{n+1}A)$ is an isomorphism for $n\geq 0$ 
thanks to \cite[Theorem 5.9]{D}, it suffices to show that $\mu_n:=\overline{\chi}\circ \Pi_{A,n}$ 
is an isomorphism from $\pi_n(U(A_\flat))$ onto $KK(A,S^{n+1}A)$ for $n\geq 0$. 
We already know from Lemma \ref{K_12} that it is the case for $n=0$. 
We denote by $\mu_{-1}$ the isomorphism from $K_0(A_\flat)$ onto $KK(A,A)$ given in 
Lemma \ref{K0}. 

Let 
$$\beta_n:\pi_n(U(A_\flat))\rightarrow \pi_{n+2}(U(A_\flat)),$$  
$$\beta_n':KK(A,S^nA)\rightarrow KK(A,S^{n+2}A),$$ 
be the isomorphisms given by the Bott periodicity. 
Setting $\pi_{-1}(U(A_\flat)):=K_0(A_\flat)$, the isomorphism   
$\beta_{-1}:K_0(A_\flat)\to K_1(SA_\flat)=\pi_1(U(A_\flat))$ also makes sense, which associates  
$K_1(u_p)$ to $K_0(p)$, where $p\in A_\flat$ is a projection and $u_p(x)=e^{2\pi i x}p+1-p$. 
To prove the statement by induction, it suffices to show that the diagram 
\begin{equation}\label{commutative}
\begin{CD}
\pi_n(U(A_\flat))@>\beta_n>> \pi_{n+2}(U(A_\flat))\\
@V{\mu_n}VV @VV{\mu_{n+2}}V\\
KK(A,S^{n+1}A)@>\beta'_{n+1}>>KK(A,S^{n+3}A)
\end{CD},
\end{equation}
is commutative for any $n\geq -1$. 
Thanks to Lemma \ref{Bott}, it is indeed the case for $n\geq 0$, and $n=-1$ is the only remaining case. 

Let $p\in A_\flat$ be a projection. 
Then $\mu_{-1}(K_0(p))=KK(\phi^p)$ and 
$$\mu_1\circ \beta_{-1}(K_0(p))=\mu_1([u_p])=KK(\psi^{u_p})-KK(j_{S^2}).$$ 
Since
$$P(s,u_p(x))=e_B(s,x)\otimes p+e_{11}\otimes (1-p),$$
we get 
$$\psi^{u_p}_t(a)(x,s)=e_B(s,x)\otimes \phi^p_t(a)+e_{11}\otimes \phi^{1-p}_t(a).$$
Thus 
$$KK(\psi^{u_p})=\beta_0'(KK(\phi^p))+KK(j_{S^2}),$$
and the above diagram is commutative for $n=-1$. 
\end{proof}

\begin{problem} Let $G$ be a discrete group and let $\alpha$ be a $G$-action on $A$. 
Then $\alpha$ induces a $G$-action on $A_\flat$, and we denote by $(A_\flat)^G$ the $G$-fixed point subalgebra of 
$A_\flat$. 
When a projection $p$ (resp. unitary $u$) is in $(A_\flat)^G$, the asymptotic morphism $\phi^p$ (resp. $\psi^u$) 
is $G$-equivariant, and it gives rise to a class in $KK^G(A,A)$ (resp. $KK^G(A,SA)$). 
With this correspondence, do we have $K_*((A_\flat)^G)\cong KK^G(A,S^*A)$ under a reasonable assumption of 
$G$ and $\alpha$ ? 
\end{problem}

\section{The $K$-groups of the central sequence algebras}
In Lemma \ref{K_12} and Lemma \ref{K0}, we computed the $K$-groups of the continuous asymptotic centralizer 
algebra $A_\flat$. 
Following the same idea, we present its discrete analogue in this section. 
Our original motivation of this research was to correct the statements in \cite[Theorem 2.1, Proposition 3.1]{K04}, 
and we would like to thank Akitaka Kishimoto for having informed us of them. 

Let $A$ be a $C^*$-algebra, and let $\omega\in \beta\N\setminus \N$ be a free ultra-filter. 
We set 
$$c_0^\omega(A)=\{(a_n)\in \ell^\infty(\N,A);\;\lim_{n\to\omega}\|a_n\|=0\},$$  
$A^\omega=A/c_0^\omega(A)$, and $A_\omega=A^\omega\cap A'$. 
Then it is known that if $A$ is a unital Kirchberg algebra, the two algebras 
$A^\omega$ and $A_\omega$ are purely infinite and simple (see \cite[Proposition 6.2.6, Proposition 7.1.1]{R2002}). 

For the definition of the $KL$-groups, see \cite[2.4.8]{R2002} and \cite[Definition 2.3]{L}. 

\begin{theorem}\label{central sequence} 
Let $A$ be a unital Kirchberg algebra and let $\omega\in \beta \N\setminus \N$ be a free ultra-filter. 
Then there exist injective group homomorphisms $\Lambda_{A,i}:K_i(A_\omega)\to KL(A,S^iA^\omega)$ for $i=0,1$. 
If moreover $A$ is in the UCT class $\cN$ of Rosenberg-Schochet \cite{RS}, 
the two maps $\Lambda_{A,0}$ and $\Lambda_{A,1}$ are isomorphisms. 
\end{theorem}

\begin{remark} When $A$ is in $\cN$, we actually have $KL(A,S^iA^\omega)=KK(A,S^iA^\omega)$ for $i=0,1$. 
Indeed, for a discrete abelian group $G$, the ultra-product $G^\omega$ is defined by  
$$G^\omega=\prod_{n\in \N}G/\bigoplus_\omega G,$$
where 
$$\bigoplus_\omega G=\{(g_n)\in \prod_{n\in \N} G;\; \{n\in\N;\;g_n=0\}\in \omega\}.$$
It is known that $G^\omega$ is always algebraically compact (see \cite[Corollary 1.12]{F}), 
and in consequence $\Pext(H,G^\omega)$ is always trivial for any $H$ (see \cite[Proposition 5.4]{S}). 
It is routine work to show $K_*(A^\omega)=K_*(A)^\omega$ if $A$ is purely infinite and simple.  
Thus the UCT implies $KL(A,S^iA^\omega)=KK(A,S^iA^\omega)$. 
\end{remark}

\begin{example} When $G$ is a finite group, we have $G^\omega=G$. 
Thus for the Cuntz algebra $\cO_n$, Theorem \ref{central sequence} and the UCT imply 
$$K_0((\cO_n)_\omega)\cong K_1((\cO_n)_\omega)\cong \Z_{n-1}$$
for finite $n$, and 
$$K_0((\cO_\infty)_\omega)\cong \Z^\omega,\quad K_1((\cO_\infty)_\omega)=\{0\}.$$
\end{example}

For the proof of Theorem \ref{central sequence}, we need to go through R{\o}rdam's $H(A,B)$ 
introduced in \cite[p.431]{R95}. 
Here we give a modified definition of it adapted to our case. 
From now on we assume that $A$ and $\omega$ are as in Theorem \ref{central sequence} 
(we do not assume that $A$ is in the UCT class $\cN$ for the moment), 
and $X$ is a compact metrizable space. 
As before, it suffices to prove Theorem \ref{central sequence} when $A$ is in the Cuntz standard form, 
and we keep this assumption throughout this section. 
Let $\Hom(A,C(X)\otimes A^\omega)$ be the set of unital homomorphisms from $A$ to $C(X)\otimes A^\omega$, 
and let $H(A,C(X)\otimes A^\omega)$ be the quotient space of $\Hom(A,C(X)\otimes A^\omega)$ by approximate unitary equivalence. 
For $\rho\in \Hom(A,C(X)\otimes A^\omega)$, we denote by $[\rho]$ its equivalence class in $H(A,C(X)\otimes A^\omega)$. 
The direct sum of $\rho,\sigma\in \Hom(A,C(X)\otimes A^\omega)$ is defined (up to equivalence) by choosing a unital copy of the Cuntz algebra 
$\cO_2$ with the standard generators  $\{S_1,S_2\}$, and setting 
$$(\rho\oplus \sigma)(a)=S_1\rho(a)S_1^*+S_2\sigma(a)S_2^*,\quad a\in A.$$

We denote by $\kappa_X:H(A,C(X)\otimes A^\omega)\to KL(A,C(X)\otimes A^\omega)$ the map associating $KL(\rho)$ to $[\rho]$. 
Then $\kappa_X$ is a semigroup homomorphism.  
Since $A^\omega$ is in the Cuntz standard form, Kirchberg's embedding theorem \cite[Theorem 6.3.11]{R2002} implies 
that there exists an element in $\Hom(A,C(X)\otimes A^\omega)$ 
factoring through $\cO_2$, which gives $0\in H(A,C(X)\otimes A^\omega)$. 

Huaxin Lin's result \cite[Theorem 3.14]{L} shows the following lemma. 

\begin{lemma} The map $\kappa_X:H(A,C(X)\otimes A^\omega)\to KL(A,C(X)\otimes A^\omega)$ is injective. 
\end{lemma}

The semigroup $H(A,A^\omega)$ is nothing but Kirchberg's $EK_\omega(A,A)$, which is actually a group thanks to  
Kirchberg's dilation lemma \cite[Corollary 4]{Kir}. 
We present an easy generalization of it here. 

\begin{lemma}\label{dilation} Let $\rho,\sigma\in \Hom(A,C(X)\otimes A^\omega$). 
Then for any $\epsilon>0$ and a finite subset $F\subset A$, there exists $v\in C(X)\otimes A^\omega$ satisfying 
$\|\rho(a)-v^*\sigma(a)v\|<\epsilon$ for any $a\in F$. 
If moreover $X=\{pt\}$, there exists an isometry $s\in A^\omega$ satisfying $\rho(a)=s^*\sigma(a)s$ for any $a\in A$.  
\end{lemma}

\begin{proof} 
Let $B=\sigma(A)\subset C(X)\otimes A^\omega$, and let $\mu=\rho\circ \sigma^{-1}$. 
Then $\mu$ is a nuclear unital homomorphism from $B$ to $C(X)\otimes A^\omega$, and $C(X)\otimes A^\omega$ is properly infinite. 
As in \cite[Proposition 6.3.3]{R2002}, we can find $v\in A^\omega$ satisfying 
$\|v^*bv-\mu(b)\|<\epsilon$ for any $b\in \sigma(F)$, and the first statement holds. 
The second statement follows from the usual diagonal sequence argument. 
\end{proof}

When $X$ has a base point $x_0$, we set 
$$\Hom(A,C(X)\otimes A^\omega)_*=\{\varphi\in \Hom (A,C(X)\otimes A^\omega);\; \forall\; a\in A,\; \varphi(a)(x_0)=a\},$$
$$H(A,C(X)\otimes A^\omega)_*=\{[\varphi]\in H(A,C(X)\otimes A^\omega);\; \varphi\in \Hom(A,C(X)\otimes A^\omega)_*\}.$$
Let $\iota_A:A\to A^\omega$ be the inclusion map, and let 
$\iota_{A,X}:A\to C(X)\otimes A^\omega$ be a homomorphism given by $\iota_{A,X}(a)=1_{C(X)}\otimes a$. 
Then $H(A,C(X)\otimes A^\omega)_*$ is the preimage of $[\iota_A]\in H(A,A^\omega)$ for the homomorphism 
$$H(A,C(X)\otimes A^\omega)\to H(A,A^\omega),$$
induced by the evaluation map at $x_0$. 

\begin{remark}\label{aue} It is easy to show the following statement: 
two homomorphisms $\rho,\sigma\in \Hom(A,C(X,A^\omega))_*$ are approximately unitarily equivalent 
in $\Hom(A,C(X,A^\omega))$ if and only if there exists a sequence of unitaries $\{u_n\}_{n=1}^\infty$ 
in $C(X,A^\omega)$ such that $u_n(x_0)=1$ for all $n\in \N$ and 
$$\lim_{n\to \infty}u_n\rho(a)u_n^*=\sigma(a),\quad \forall a\in A.$$
\end{remark}

Since $\kappa_X$ is injective, homotopic homomorphisms in $\Hom(A,C(X)\otimes A^\omega)$ give the same class in $H(A,C(X)\otimes A^\omega)$. 
Using this fact, we can introduce a group structure into $H(A,C(\Sigma X)\otimes A^\omega)_*$ as in \cite[p.269]{Bl}, 
though it is not a subsemigroup of $H(A,C(\Sigma X)\otimes A^\omega)$ with its original addition coming from a direct sum. 

For $\rho,\sigma\in \Hom(A,C(\Sigma X,A^\omega))_*$, we define $\rho\uplus \sigma\in \Hom(A,C(S^1,A^\omega))_*$ by  
$$(\rho\uplus \sigma)(a)(x,t)=\left\{
\begin{array}{ll}
\rho(a)(x,2t) , &\quad  t\in [0,1/2]\\
\sigma(a)(x,2t-1) , &\quad t\in [1/2,1]
\end{array}
\right..$$
Then the operation $\uplus$ on $H(A,C(\Sigma X)\otimes A^\omega)_*$ given by $[\rho]\uplus[\sigma]:=[\rho\uplus\sigma]$ 
is well-defined, and we get a group $(H(A,C(\Sigma X)\otimes A^\omega)_*,\uplus)$ with unit $[\iota_{A,\Sigma X}]$. 

\begin{lemma}\label{sum} The map 
$$S\kappa_{X}:H(A,C(\Sigma X)\otimes A^\omega)_*\to KL(A,SC(X)\otimes A^\omega)$$ 
defined by $S\kappa_{X}([\rho])=\kappa_{\Sigma X}([\rho])-\kappa_{\Sigma X}([\iota_{A,\Sigma X}])$ 
is an injective group homomorphism from $(H(A,C(\Sigma X)\otimes A^\omega)_*,\uplus)$ to $KL(A,SC(X)\otimes A^\omega)$. 
If moreover $\kappa_{\Sigma X}$ is an isomorphism, so is $S\kappa_{X}$. 
\end{lemma}

\begin{proof} Let $\rho,\sigma \in \Hom(A,C(\Sigma X)\otimes A^\omega)_*$. 
We claim 
$$KK(\rho\uplus \sigma)+KK(\iota_{A,\Sigma X})=KK(\rho)+KK(\sigma),$$
which is probably well-known, but we include the proof here for the reader's convenience. 
Let $\mu,\rho', \sigma'\in \Hom(A,C(\Sigma X)\otimes A^\omega)_*$ be given by 
$$\mu(a)(x,t)=\left\{
\begin{array}{ll}
 \rho(a)(x,4t), &\quad  t\in [0,1/4]\\
 a, &\quad  t\in [1/4,2/4]\\
\sigma(a)(x,4t-2) , &\quad t\in [1/2,3/4]\\
a,&\quad t\in [1/4,1]
\end{array}
\right., 
$$
$$\rho'(a)(x,t)=\left\{
\begin{array}{ll}
\rho(a)(x,4t) , &\quad t\in [0,1/4] \\
a , &\quad t\in [1/4,1]
\end{array}
\right.,
$$
$$\sigma'(a)(x,t)=\left\{
\begin{array}{ll}
a , &\quad t\in [0,1/2]\cup [3/4,1] \\
\sigma(a)(x,4t-2) , &\quad t\in [1/2,3/4]
\end{array}
\right..
$$
Then $\mu$ is homotopic to $\rho \uplus \sigma$, $\rho'$ is homotopic to $\rho$, and $\sigma'$ is homotopic to $\sigma$.  

We choose isometries $S_1,S_2\in A$ satisfying the $\cO_2$ relation, and set 
$$(\mu\oplus \iota_{A,\Sigma X})(a)=S_1\mu(a)S_1^*+S_2aS_2^*,$$
$$(\rho'\oplus \sigma')(a)=S_1\rho'(a)S_1^*+S_2\sigma'(a)S_2^*.$$
Let $c(t)=\cos(\pi t/2)$, $s(t)=\sin(\pi t/2)$. 
For $t\in [0,1]$, we set 
$$T_1(t)=\left\{
\begin{array}{ll}
S_1 , &\quad  t\in [0,1/4]\\
c(4t-1)S_1+s(4t-1)S_2 , &\quad t\in [1/4,1/2] \\
S_2 , &\quad t\in [1/2,3/4]\\
s(4t-3)S_1+c(4t-3)S_2,&\quad  t\in [3/4,1]
\end{array}
\right., 
$$ 
$$T_2(t)=\left\{
\begin{array}{ll}
S_2 , &\quad  t\in [0,1/4]\\
c(4t-1)S_2-s(4t-1)S_1 , &\quad t\in [1/4,1/2] \\
-S_1 , &\quad t\in [1/2,3/4]\\
s(4t-3)S_2-c(4t-3)S_1,&\quad  t\in [3/4,1]
\end{array}
\right.. 
$$ 
Then $T_1$ and $T_2$ are isometries in $C(\Sigma X,A)$ satisfying the $\cO_2$ relation. 
Let 
$$u=T_1S_1^*+T_2S_2^*\in U(C(\Sigma X,A)).$$ Then we get 
$u(\mu\oplus \iota_{A,\Sigma X})(a)u^*=(\rho'\oplus \sigma')(a)$, which shows the claim.  

The claim implies  
$$KL(\rho\uplus \sigma)-KL(\iota_{A,\Sigma X})=(KL(\rho)-KL(\iota_{A,\Sigma X}))+(KL(\sigma)-KL(\iota_{A,\Sigma X})),$$
which shows that $S\kappa_{X}$ is a group homomorphism. 
Since $\kappa_{\Sigma X}$ is injective, so is $S\kappa_{X}$, and if $\kappa_{\Sigma X}$ is surjective, so is $S\kappa_{X}$ too. 
\end{proof}

Let $p\in A_\omega$ be a non-zero projection. 
Since $A^\omega\cap \{p\}'$ has a unital copy of the Cuntz algebra $\cO_2$, its $K_0$-class in $K_0(A^\omega)$ is trivial. 
Since $K_0(p)=K_0(1)=0$ in $K_0(A^\omega)$ and $A^\omega$ is purely infinite simple, 
there exists an isometry $v\in A^\omega$ whose range projection is $p$. 
Then $\psi^v(a)=v^*av$ is an element of $\Hom (A,A^\omega)$. 
As in the proof of Lemma \ref{K0}, we can show that the map $\nu_0:K_0(A_\omega)\to H(A,A^\omega)$ defined by 
$\nu_0(K_0(p))=[\psi^v]$ is actually a well-defined injective map.

\begin{lemma} The map $\nu_0:K_0(A_\omega)\to H(A,A^\omega)$ is a group isomorphism. 
\end{lemma}

\begin{proof} Lemma \ref{dilation} shows that $\nu_0$ is surjective, and it suffices to show that $\nu_0$ is a group homomorphism. 
Let $p_1,p_2\in A_\omega$ be non-zero projections, and let $v_1,v_2\in A^\omega$ be isometries satisfying $v_iv_i^*=p_i$, $i=1,2$. 
We choose a unital copy of the Cuntz algebra $\cO_2$ in $A$ with the canonical generators $\{S_1,S_2\}$, and set 
$$\rho(a)=S_1\psi^{v_1}(a)S_1^*+S_2\psi^{v_2}(a)S_2^*.$$
By construction, we have $[\rho]=\nu_0(K_0(p_1))+\nu_0(K_0(p_2))$. 
We choose two isometries $T_1,T_2\in A_\omega$ with mutually orthogonal ranges, and set 
$$w=T_1v_1S_1^*+T_2v_2S_2^*\in A^\omega.$$
Then $w$ is an isometry satisfying $w^*aw=\rho(a)$ for any $a\in A$, and 
$$ww^*=T_1p_1T_1^*+T_2p_2T_2^*,$$
which shows $[\rho]=\nu_0(K_0(ww^*))=\nu_0(K_0(p_1)+K_0(p_2))$. 
\end{proof}

Let $u\in U(A_\omega)$. 
Since $A^\omega\cap \{u\}'$ contains a unital copy of the Cuntz algebra $\cO_2$, 
the $K_1$-class of $u$ in $K_1(A^\omega)$ is trivial. 
Thus we can choose a continuous path $u(t)$ in $U(A^\omega)$ with $u(0)=1$ and $u(1)=u$. 
Then $\rho^u(a)(t)=u(t)au(t)^*$ gives an element in $\Hom(A,C(S^1,A^\omega))_*$. 
If $\{u'(t)\}_{t\in [0,1]}$ is another continuous path from 1 to $u$ in $U(A^\omega)$, 
we have $u'(t)au'(t)^*=\Ad w(t)\circ \rho^u(a)$ with $w(t)=u'(t)u(t)^*$. 
Since $w\in \Map(S^1,U(A^\omega))$, the class $[\rho^u]\in H(A,C(S^1,A^\omega))_*$ 
does not depend on the choice of the path $\{u(t)\}_{t\in [0,1]}$. 

Assume $v\in U(A_\omega)$ with $K_1(u)=K_1(v)$. 
Since $A_\omega$ is purely infinite simple, there exist self-adjoint $a,b\in A_\omega$ with $0\leq a,b <2\pi$ 
satisfying $v=ue^{ia}e^{ib}$. 
We can choose a path $v(t)=u(t)e^{ita}e^{itb}$ connecting 1 and $v$ in $U(A^\omega)$, and with this choice, we get 
$\rho^v=\rho^u$. 
Thus $\nu_1(K_1(u))=[\rho^u]$ gives a well-defined map 
$$\nu_1:K_1(A_\omega)\to H(A,C(S^1)\otimes A^\omega)_*.$$

\begin{lemma} The map $\nu_1$ from $K_1(A_\omega)$ to $H(A,C(S^1)\otimes A^\omega)_*$ is an injective group homomorphism. 
\end{lemma}

\begin{proof} We first show that $\nu_1$ is a group homomorphism. 
Let $u,v\in U(A_\omega)$. 
We choose continuous paths $\{u(t)\}_{t\in [0,1]}$ and $\{v(t)\}_{t\in[0,1]}$ in $U(A^\omega)$ satisfying 
$u(0)=1$, $u(t)=u$ for $t\in [1/2,1]$, $v(t)=1$ for $[0,1/2]$, and $v(1)=v$. 
Then $\{u(t)v(t)\}_{t\in [0,1]}$ is a continuous path from 1 to $uv$ in $U(A^\omega)$, and 
$[\rho^{uv}]=[\rho^u\uplus\rho^v]$. 
Thus $\nu_1$ is a group homomorphism. 

Next we show that $\nu_1$ is injective. 
Assume $u\in U(A_\omega)$ with $\nu_1(K_1(u))=0$. 
We choose a continuous path $\{u(t)\}_{t\in [0,1]}$ in $U(A^\omega)$, and define $\rho^u$ as before. 
Then $\rho^u\in \Hom(A,C(S^1,A^\omega))_*$ with $[\rho^u]=[\iota_{A,S^1}]$, and as noted in Remark \ref{aue}, 
there exists a sequence of unitaries 
$\{v_n\}_{n=1}^\infty$ in $C(S^1,A^\omega)$ such that $v_n(0)=1$ and 
$$\lim_{n\to\infty}\sup_{t\in [0,1]}\|v_n(t)u(t)xu(t)^*v_n(t)^*-x\|=0,\quad \forall x\in A.$$
This with Nakamura's Theorem \cite[Theorem 7]{N} shows that for any positive number $\epsilon$ and any finite 
subset $F\subset A$, there exists a continuous path $\{w(t)\}_{t\in [0,1]}$ in $U(A^\omega)$ such that 
$w(0)=1$, $w(1)=u$, $\|[w(t),x]\|<\varepsilon$ for any $t\in [0,1]$, $x\in F$, and $\mathrm{Lip}(w)<6\pi$. 
Hence, a diagonal sequence argument shows $u$ is connected to 1 in $U(A_\omega)$, and 
$K_1(u)=0$ in $K_1(A_\omega)$. 
\end{proof}

So far we have proved the first statement of Theorem \ref{central sequence}. 
To prove the second one, it suffices to show that $\kappa_{X}$ and $\nu_1$ are surjective under 
the assumption that $A$ is in the UCT class $\cN$. 

We denote by $\underline{K}(A)$ the entire $K$-group of $A$, and by $\Hom_\Lambda(\underline{K}(A),\underline{K}(B))$ 
the set of graded homomorphisms from $\underline{K}(A)$ to $\underline{K}(B)$ preserving the action of the category $\Lambda$ 
(see \cite{DL} for the definition). 
 
\begin{lemma} \label{separable} With the notation above, we assume that $A$ is in the UCT class $\cN$. 
Let 
$$h\in \Hom_\Lambda(\underline{K}(A),\underline{K}(C(X)\otimes A^\omega)).$$ 
Then there exist a unital simple separable $C^*$-subalgebra $B$ of $A^\omega$ and a homomorphism 
$$k\in \Hom_\Lambda(\underline{K}(A),\underline{K}(C(X)\otimes B))$$
satisfying $B\cong B\otimes \cO_\infty$ and  
$h=\underline{K}(\iota)\circ k$, where 
$\iota:C(X)\otimes B\rightarrow C(X)\otimes A^\omega$ is the inclusion map. 
\end{lemma}

\begin{proof} Since $\underline{K}(A)$ is a countable group and $X$ is a compact metrizable space, 
there exists a unital separable $C^*$-subalgebra $B_0$ of $A^\omega$ and a homomorphism 
$$k_0\in \Hom_\Lambda(\underline{K}(A),\underline{K}(C(X)\otimes B_0))$$
satisfying $h=\underline{K}(\iota_0)\circ k$, where 
$\iota_0:C(X)\otimes B_0\rightarrow C(X)\otimes A^\omega$ is the inclusion map. 
To prove the statement, it suffices to show that there exists a separable simple $C^*$-subalgebra $B$ of $A^\omega$ 
containing $B_0$ and satisfying $B\cong B\otimes \cO_\infty$.

Since $A^\omega$ is unital simple purely infinite, we can inductively construct an increasing sequence of separable 
$C^*$-subalgebras $\{B_n\}_{n=0}^\infty$ of $A^\omega$ such that for any non-zero positive element $a\in B_n$, there exists $b\in B_{n+1}$ 
satisfying $b^*ab=1_{A^\omega}$. 
Then the closure $B_\infty$ of $\cup_{n}B_n$ is simple.  
Since $B_\infty$ is separable, we can find a unital embedding $\theta:\cO_\infty\to A^\omega\cap B_\infty'$,  
and the $C^*$-algebra $B$ generated by $B_\infty$ and $\theta(\cO_\infty)$ is isomorphic to $B_\infty\otimes \cO_\infty$, 
and hence $B\cong B\otimes \cO_\infty$. 
\end{proof}

\begin{lemma}\label{H=KL} Let the notation be as above. 
If $A$ is in the UCT class $\cN$, the map $\kappa_X:H(A,C(X)\otimes A^\omega))\to KL(A,C(X)\otimes A^\omega))$ is an isomorphism. 
\end{lemma}

\begin{proof} It suffices to show that $\kappa_X$ is surjective. 
Recall that Dadarlat-Loring's UMCT \cite{DL} implies   
$$KL(A,C(X)\otimes A^\omega)\cong \Hom_\Lambda(\underline{K}(A),\underline{K}(C(X)\otimes A^\omega)).$$

Let $x\in KL(A,C(X)\otimes A^\omega))$ be given, and let 
$h\in  \Hom_\Lambda(\underline{K}(A),\underline{K}(C(X)\otimes A^\omega))$ be the corresponding homomorphism. 
Then thanks to Lemma \ref{separable}, there exist a unital simple separable $C^*$-subalgebra $B$ of 
$A^\omega$ and  $k\in \Hom_\Lambda(\underline{K}(A),\underline{K}(C(X)\otimes B))$ satisfying $B\cong B\otimes \cO_\infty$, and 
$h=\underline{K}(\iota)\circ k$, where 
$\iota:C(X)\otimes B\rightarrow C(X)\otimes A^\omega$ is the inclusion map. 
Thanks to the UMCT exact sequence, we can choose $\tilde{k}\in KK(A,C(X)\otimes B)$ giving rise to $k$. 
Thus the Kirchberg-Phillips classification theorem implies that there exists $\rho\in \Hom(A,C(X)\otimes B)$ 
satisfying $\tilde{k}=KK(\rho)$, and so 
$\underline{K}(\iota\circ \rho)=h$. 
This implies $KL(\iota\circ \rho)=x$, and $\kappa_X$ is surjective. 
\end{proof}

\begin{lemma} \label{criterion} Assume that $A$ is in the UCT class $\cN$. 
Let $\rho,\sigma \in \Hom(A,C(S^1,A^\omega))_*$, and let 
$\gamma:C(S^1)\otimes A^\omega\rightarrow C(S^1,A)^\omega$ be the inclusion map. 
If $\gamma\circ \rho$ and $\gamma\circ \sigma$ are unitarily equivalent, 
then $[\rho]=[\sigma]$ in $H(A,C(S^1)\otimes A^\omega)_*$.
\end{lemma}

\begin{proof} It suffices to show that $\underline{K}(\gamma)$ is an isomorphism because 
$[\rho]=[\sigma]$ is equivalent to $\underline{K}(\rho)=\underline{K}(\sigma)$. 
First we claim $K_*(\gamma)$ is an isomorphism for $*=0,1$. 
Indeed, this follows from 
$$K_*(C(S^1)\otimes A^\omega)\cong K_0(A^\omega)\oplus K_1(A^\omega)\cong K_0(A)^\omega\oplus K_1(A)^\omega,$$
$$K_*(C(S^1,A)^\omega)=(K_0(A)\oplus K_1(A))^\omega\cong K_0(A)^\omega\oplus K_1(A)^\omega.$$
Now, the statement follows from the K{\"u}nneth theorem. 
\end{proof}

\begin{lemma} Under the assumption that $A$ is in the UCT class $\cN$, the map 
$\nu_1:K_1(A_\omega)\to H(A,C(S^1)\otimes A^\omega)_*$ is an isomorphism. 
\end{lemma}

\begin{proof} It suffices to show that $\nu_1$ is surjective. 
Let $\sigma\in \Hom(A,C(S^1)\otimes A^\omega)_*$. 
We first claim that there exists a unitary $w\in U(C([0,1],A)^\omega)$ satisfying 
$\sigma(x)=wxw^*$ for all $x\in A$, where 
we emphasize that the equality is understood in $C([0,1],A)^\omega$. 
Indeed, we consider $[0,1]$ as a based space with a base point 0, and treat $\sigma$ as an element 
of $\Hom(A,C([0,1],A^\omega))_*$. 
Then $\sigma$ is homotopic to $\iota_{A,[0,1]}$, and we have $KL(\sigma)=KL(\iota_{A,[0,1]})$ 
in $KL(A,C([0,1],A^\omega))$. 
Thus \cite[Theorem 3.14]{L} implies that $\sigma$ and $\iota_{A,X}$ are approximately unitarily equivalent. 
By the usual diagonal sequence argument, we get $w$.

Let $(w_n)$ be a representing sequence of $w$. 
By replacing $w_n(t)$ with $w_n(t)w_n(0)^{-1}$ if necessary, we may and do assume $w_n(0)=1$. 
Let $u_n=w_n(1)$. 
Then $u=(u_n)$ belongs to $U(A_\omega)$.  
We take a continuous path $\{u(t)\}_{t\in [0,1]}$ of unitaries in $A^\omega$ such that $u(0)=1$ and 
$u(1)=u$, and let $\rho(x)(t)=u(t)xu(t)^*$ for $x\in A$. 
We show $[\sigma]=[\rho]$ in $H(A,C(S^1)\otimes A^\omega)_*$. 
Indeed, let $\gamma:C(S^1)\otimes A^\omega\rightarrow C(S^1,A)^\omega$ be the inclusion map. 
Then by construction $\gamma\circ \sigma$ and $\gamma\circ \rho$ are unitarily equivalent. 
Therefore thanks to Lemma \ref{criterion}, we conclude $[\sigma]=[\rho]$ in $H(A,C(S^1)\otimes A^\omega)_*$. 
This shows that $\mu$ is a surjection. 
\end{proof}


\section{Appendix}
In this appendix, we establish a refinement of Nakamura's homotopy theorem (Theorem \ref{refinedNakamura}), 
which is used in  Lemma \ref{K_12}. 
We formulate the statement involving group actions as we need it in our companion papers \cite{IM1}, \cite{IM2}.

We first recall Haagerup-R\o rdam's construction 
\cite[Lemma 5.1]{HR95} of a continuous path of unitaries. 
Consider the Cuntz algebra $\cO_2=C^*(S_1,S_2)$ and let 
\[
R_1=S_1,\quad R_2=S_2S_1,\quad R_3=S_2S_2, 
\]
\[
T_1=S_2,\quad T_2=S_1S_1,\quad T_3=S_1S_2. 
\]
Then $\{R_1,R_2,R_3\}$ and $\{T_1,T_2,T_3\}$ satisfy 
the $\mathcal{O}_3$-relation. 
Let $\rho$ and $\tau$ be 
the unital homomorphisms from $M_3(\C)$ to $\mathcal{O}_2$ given by 
\[
\rho(e_{ij})=R_iR_j^*,\quad \tau(e_{ij})=T_iT_j^*, 
\]
where $(e_{ij})_{ij}$ is the canonical system of matrix units in $M_3(\C)$. 
We have 
\[
\rho(e_{11})=\tau(e_{22}+e_{33}),\quad \tau(e_{11})=\rho(e_{22}+e_{33}). 
\]
Let 
\[
w=\begin{bmatrix}0&0&1\\1&0&0\\0&1&0\end{bmatrix}\in M_3(\C)
\]
and let $h\in M_3(\C)$ be the self-adjoint matrix 
such that $e^{2\pi\sqrt{-1}h}=w$ and $\lVert h\rVert=1/3$. 

Let $A$ be a (unital or non-unital) $C^*$-algebra and let $u\in U(A)$. 
We construct a continuous map $\hbar(u):[0,1]\to U(A\otimes\mathcal{O}_2)$ 
from $1$ to $u\otimes1$ as follows. 
Set 
\[
\tilde u=\begin{bmatrix}u^*&0&0\\0&1&0\\0&0&u\end{bmatrix}
\in U(A\otimes M_3(\C))
\]
and set 
\[
\tilde u(t)=(1\otimes e^{2\pi\sqrt{-1}th})\tilde u
(1\otimes e^{-2\pi\sqrt{-1}th})\tilde u^*. 
\]
We define 
\[
\hbar(u)(t)
=(\id\otimes\rho)(\tilde u(t))(\id\otimes\tau)(\tilde u(t))
\in U(A\otimes\mathcal{O}_2). 
\]

\begin{lemma}[{\cite[Lemma 5.1]{HR95}}]\label{HaagerupRordam}
Let $A$ be a $C^*$-algebra and let $u\in U(A)$. 
The continuous path $\hbar(u):[0,1]\to U(A\otimes\mathcal{O}_2)$ 
constructed above satisfies the following. 
\begin{enumerate}
\item [$(1)$] $\hbar(u)(0)=1$, $\hbar(u)(1)=u\otimes1$ and 
$\Lip(\hbar(u))\leq8\pi/3$. 
\item [$(2)$] $\lVert[\hbar(u)(t),a\otimes1]\rVert\leq4\lVert[u,a]\rVert$ holds 
for any $t\in[0,1]$ and $a\in A$. 
\item [$(3)$] Let $v\in U(A)$ be another unitary. 
Then $\lVert\hbar(u)(t)-\hbar(v)(t)\rVert\leq4\lVert u-v\rVert$. 
\item [$(4)$] If $u$ is in $U(B)$ for some subalgebra $B\subset A$, 
then $\hbar(u)(t)$ is in $U(B\otimes\mathcal{O}_2)$. 
\end{enumerate}
\end{lemma}
\begin{proof}
(1) and (2) are consequences of \cite[Lemma 5.1]{HR95}. 
(3) and (4) are immediate from the construction. 
\end{proof}

\begin{lemma}\label{hmtpyofhomo}
Let $A$ be a unital $C^*$-algebra such that $U(A)=U(A)_0$. 
Let $\sigma_1$ and $\sigma_2$ be 
unital homomorphisms from $\mathcal{O}_2$ to $A$. 
Then there exists a homotopy of unital homomorphisms $(\rho_t)_{t\in[0,1]}$ 
between $\sigma_1$ and $\sigma_2$. 
\end{lemma}
\begin{proof}
We can define a unitary $u\in U(A)$ by 
\[
u=\sigma_2(S_1)\sigma_1(S_1)^*+\sigma_2(S_2)\sigma_1(S_2)^*. 
\]
By $U(A)=U(A)_0$, there exists a continuous path $u:[0,1]\to U(A)$ 
such that $u(0)=1$ and $u(1)=u$. 
Now $\rho_t:\mathcal{O}_2\to A$ defined by 
\[
\rho_t(S_1)=u(t)\sigma_1(S_1),\quad \rho_t(S_2)=u(t)\sigma_1(S_2)
\]
gives a desired homotopy. 
\end{proof}

\begin{lemma}\label{explog}
Let $A$ be a unital $C^*$-algebra. 
If a unitary $u\in A$ satisfies $\lVert u-1\rVert<1$, then 
\[
\lVert[\exp(s\log u),a]\rVert
\leq\frac{1}{1-\lVert u{-}1\rVert}\lVert[u,a]\rVert
\]
holds for any $a\in A$ and $s\in[0,1]$. 
\end{lemma}
\begin{proof}
Let $f(t)=\exp(s\log(1{+}t))$. 
We have $(1{+}t)f'(t)=sf(t)$, 
and so $(1{+}t)f^{(n+1)}(t)+nf^{(n)}(t)=sf^{(n)}(t)$. 
Therefore, we get $f(0)=1$, $f'(0)=s$ and 
\[
f^{(n)}(0)=s(s-1)(s-2)\dots(s-(n{-}1)). 
\]
Because $\lvert f^{(n)}(0)\rvert\leq(n{-}1)!$ for $n\geq2$, one obtains 
\begin{align*}
\lVert[\exp(s\log u),a]\rVert
&=\lVert[f(u{-}1),a]\rVert\\
&\leq\sum_{n=0}^\infty\frac{\lvert f^{(n)}(0)\rvert}{n!}
\lVert[(u{-}1)^n,a]\rVert\\
&\leq\sum_{n=0}^\infty\frac{\lvert f^{(n)}(0)\rvert}{n!}
\cdot n\lVert u-1\rVert^{n-1}\lVert[u,a]\rVert\\
&\leq\sum_{n=1}^\infty\lVert u-1\rVert^{n-1}\lVert[u,a]\rVert
=\frac{1}{1-\lVert u{-}1\rVert}\lVert[u,a]\rVert. 
\end{align*}
\end{proof}

By using Lemma \ref{HaagerupRordam}, Lemma \ref{hmtpyofhomo} 
and Lemma \ref{explog}, 
we prove the following proposition. 

\begin{proposition}
Let $A$ be a unital $C^*$-algebra and let $B\subset A$ be a subalgebra. 
Let $x:[0,1]\times[0,\infty)\to U(B)$ be a continuous map. 
Then there exists a continuous map 
$y:[0,1]\times[0,\infty)\to U(B\otimes\mathcal{O}_\infty)$ such that 
\[
y(0,t)=x(0,t)\otimes1,\quad y(1,t)=x(1,t)\otimes1,\quad 
\Lip(y(\cdot,t))\leq24\pi
\quad\forall t\in[0,\infty)
\]
and 
\[
\limsup_{t\to\infty}\sup_{s\in[0,1]}\lVert[y(s,t),a\otimes1]\rVert
\leq243\limsup_{t\to\infty}\sup_{s\in[0,1]}\lVert[x(s,t),a]\rVert
\quad\forall a\in A. 
\]
\end{proposition}
\begin{proof}
Since $x$ is uniformly continuous on $[0,1]\times[n,n{+}1]$, 
we can choose a strictly increasing sequence of natural numbers $(m_n)_n$ 
such that $\lVert x(s_1,t)-x(s_2,t)\rVert\leq1/4$ holds 
for all $(s_1,t),(s_2,t)\in[0,1]\times[n,n{+}1]$ 
with $\lvert s_1-s_2\rvert\leq1/m_n$. 
We may assume that $m_n$ are odd numbers and $m_n$ divides $m_{n+1}$. 
For each $n$, we choose a partition of unity of $\mathcal{O}_\infty$ 
consisting of projections $\{e^{(n)}_1,e^{(n)}_2,\dots,e^{(n)}_{m_n}\}$ 
so that 
\[
K_0(e^{(n)}_j)=\begin{cases}K_0(1_{\mathcal{O}_\infty})&\text{$j$ is odd}\\
-K_0(1_{\mathcal{O}_\infty})&\text{$j$ is even}\end{cases}
\quad\forall j=1,2,\dots,m_n
\]
and 
\[
e^{(n)}_j=\sum_{i=(j-1)k_n+1}^{jk_n}e^{(n+1)}_i
\quad\forall j=1,2,\dots,m_n, 
\]
where $k_n=m_{n+1}/m_n$. 
The class of $e^{(n)}_j+e^{(n)}_{j+1}$ is trivial 
in $K_0(\mathcal{O}_\infty)$, and so 
there exists a unital homomorphism 
\[
\sigma^{(n)}_j:\mathcal{O}_2
\to\left(e^{(n)}_j+e^{(n)}_{j+1}\right)\mathcal{O}_\infty
\left(e^{(n)}_j+e^{(n)}_{j+1}\right)
\]
for $j=1,2,\dots,m_n{-}1$. 
Let 
\[
I_{j,n}=\{(j{-}1)k_n+1+2l\mid l=0,1,\dots,k_n{-}1\}. 
\]
One has 
\[
\sigma^{(n)}_j(1)=e^{(n)}_j+e^{(n)}_{j+1}
=\sum_{i=(j-1)k_n+1}^{(j+1)k_n}e^{(n+1)}_i
=\sum_{i\in I_{j,n}}\sigma^{(n+1)}_i(1). 
\]
For any non-zero projection $e\in\mathcal{O}_\infty$, 
the corner algebra $e\mathcal{O}_\infty e$ is purely infinite and simple, 
and $K_1(e\mathcal{O}_\infty e)=K_1(\mathcal{O}_\infty)=0$. 
Hence $U(e\mathcal{O}_\infty e)=U(e\mathcal{O}_\infty e)_0$. 
Therefore, by Lemma \ref{hmtpyofhomo}, 
there exists a continuous family $(\rho^{(n)}_{j,t})_{t\in[n,n+1]}$ 
of unital homomorphisms from $\mathcal{O}_2$ to 
$\left(e^{(n)}_j+e^{(n)}_{j+1}\right)\mathcal{O}_\infty
\left(e^{(n)}_j+e^{(n)}_{j+1}\right)$ such that 
\[
\rho^{(n)}_{j,n}=\sigma^{(n)}_j,\quad 
\rho^{(n)}_{j,n+1}=\bigoplus_{i\in I_{j,n}}\sigma^{(n+1)}_i. 
\]

We construct a continuous map 
\[
v_n:([0,1/3]\cup[2/3,1])\times[n,n{+}1]\to U(B\otimes\mathcal{O}_\infty)
\]
as follows. 
Let $J_{0,n}=\{2,4,\dots,m_n{-}1\}$ and 
$J_{1,n}=\{1,3,\dots,m_n{-}2\}$. 
For $(s,t)\in[0,1/3]\times[n,n{+}1]$, we set 
\[
v_n(s,t)=x(0,t)\otimes e^{(n)}_1
+\sum_{j\in J_{0,n}}(\id\otimes\rho^{(n)}_{j,t})
\bigl((x(0,t)\otimes1)\hbar(x(0,t)^*x(j/m_n,t))(3s)\bigr). 
\]
For $(s,t)\in[2/3,1]\times[n,n{+}1]$, we set 
\[
v_n(s,t)=\sum_{j\in J_{1,n}}(\id\otimes\rho^{(n)}_{j,t})
\bigl((x(1,t)\otimes1)\hbar(x(1,t)^*x(j/m_n,t))(3{-}3s)\bigr)
+x(1,t)\otimes e^{(n)}_{m_n}. 
\]
By Lemma \ref{HaagerupRordam} (3) and (4), these maps are well-defined. 
By Lemma \ref{HaagerupRordam} (2), one has 
\[
\lVert[v_n(s,t),a\otimes1]\rVert
\leq9\sup_{r\in[0,1]}\lVert[x(r,t),a]\rVert\quad\forall a\in A. 
\]
By Lemma \ref{HaagerupRordam} (1), 
we also have $\Lip(v_n(\cdot,t))\leq8\pi$, $v_n(0,t)=x(0,t)$, 
\[
v_n(1/3,t)=x(0,t)\otimes e^{(n)}_1
+\sum_{j\in J_{0,n}}x(j/m_n,t)\otimes\left(e^{(n)}_j+e^{(n)}_{j+1}\right), 
\]
\[
v_n(2/3,t)=\sum_{j\in J_{1,n}}x(j/m_n,t)\otimes
\left(e^{(n)}_j+e^{(n)}_{j+1}\right)+x(1,t)\otimes e^{(n)}_{m_n}
\]
and $v_n(1,t)=x(1,t)$. 
By the choice of $m_n$, we get 
\[
\left\lVert v_n(1/3,t)-v_n(2/3,t)\right\rVert\leq1/4. 
\]

Next, we would like to patch $(v_n)_n$ together and construct 
$y:([0,1/3]\cup[2/3,1])\times[0,\infty)\to U(B\otimes\mathcal{O}_\infty)$. 
For $s\in[0,1/3]$, we have 
\begin{align*}
v_n(s,n{+}1)&=x(0,n{+}1)\otimes e^{(n)}_1\\
&\! +\sum_{j\in J_{0,n}}\sum_{i\in I_{j,n}}
(\id\otimes\sigma^{(n+1)}_i)
\bigl((x(0,n{+}1)\otimes1)\hbar(x(0,n{+}1)^*x(j/m_n,n{+}1))(3s)\bigr)
\end{align*}
and 
\begin{align*}
v_{n+1}(s,n{+}1)&=x(0,n{+}1)\otimes e^{(n+1)}_1\\
&\! +\sum_{i\in J_{0,n+1}}(\id\otimes\sigma^{(n+1)}_i)
\bigl((x(0,n{+}1)\otimes1)\hbar(x(0,n{+}1)^*x(i/m_{n+1},n{+}1))(3s)\bigr). 
\end{align*}
For each $i\in J_{0,n+1}$, we define $j_i\in\N\cup\{0\}$ as follows. 
For $i=2,4,\dots,k_n{-}1$, let $j_i=0$. 
In this case $i/m_{n+1}<k_n/m_{n+1}=1/m_n$. 
When $k_n<i<m_{n+1}$, let $j_i\in J_{0,n}$ be the number 
satisfying $(j_i{-}1)k_n{+}1\leq i\leq(j_i{+}1)k_n{-}1$. 
Then one has
\[
\frac{j_i{-}1}{m_n}+\frac{1}{m_{n+1}}\leq\frac{i}{m_{n+1}}
\leq\frac{j_i{+}1}{m_n}-\frac{1}{m_{n+1}}, 
\]
which implies 
\[
\left\lvert\frac{i}{m_{n+1}}-\frac{j_i}{m_n}\right\rvert<\frac{1}{m_n}. 
\]
Define continuous functions $f_i:[n,n{+}1]\to[0,1]$ by 
\[
f_i(t)=\begin{cases}j_i/m_n&t=n\\
\text{linear}&n<t<n+1\\ i/m_{n+1}&t=n{+}1. \end{cases}
\]
For $(s,t)\in[0,1/3]\times[n,n{+}1]$, we set 
\begin{align*}
w_n(s,t)&=x(0,n{+}1)\otimes e^{(n+1)}_1\\
&\quad+\sum_{i\in J_{0,n+1}}(\id\otimes\sigma^{(n+1)}_i)
\bigl((x(0,n{+}1)\otimes1)\hbar(x(0,n{+}1)^*x(f_i(t),n{+}1))(3s)\bigr). 
\end{align*}
Then $w_n(s,n)=v_n(s,n{+}1)$, $w_n(s,n{+}1)=v_{n+1}(s,n{+}1)$, 
$\Lip(w_n(\cdot,t))\leq8\pi$ and 
\[
\lVert[w_n(s,t),a\otimes1]\rVert
\leq9\sup_{r\in[0,1]}\lVert[x(r,n{+}1),a]\rVert\quad\forall a\in A. 
\]
Moreover, we can see $\lVert w_n(1/3,t)-v_n(1/3,n{+}1)\rVert\leq1/4$. 
Define a continuous map 
$y:[0,1/3]\times[0,\infty)\to U(B\otimes\mathcal{O}_\infty)$ by 
\[
y(s,t)=w_n(s,t)v_n(s,n{+}1)^*v_n(s,t)
\]
for $(s,t)\in[0,1/3]\times[n,n{+}1]$. 
This is well-defined, and $y(0,t)=v_n(0,t)=x(0,t)$, 
$\Lip(y(\cdot,t))\leq24\pi$ and 
\[
\limsup_{t\to\infty}
\sup_{s\in[0,1/3]}\lVert[y(s,t),a\otimes1]\rVert
\leq27\limsup_{t\to\infty}\sup_{r\in[0,1]}\lVert[x(r,t),a]\rVert
\quad\forall a\in A. 
\]
Furthermore, we get $\lVert y(1/3,t)-v_n(1/3,t)\rVert\leq1/4$. 

In the same way, we can construct 
$w_n:[2/3,1]\times[n,n{+}1]\to U(B\otimes\mathcal{O}_\infty)$ and 
define $y:[2/3,1]\times[0,\infty)\to U(B\otimes\mathcal{O}_\infty)$ 
satisfying $y(1,t)=x(1,t)$, $\Lip(y(\cdot,t))\leq24\pi$ and 
\[
\limsup_{t\to\infty}
\sup_{s\in[2/3,1]}\lVert[y(s,t),a\otimes1]\rVert
\leq27\limsup_{t\to\infty}\sup_{r\in[0,1]}\lVert[x(r,t),a]\rVert
\quad\forall a\in A. 
\]
Furthermore, one has $\lVert y(2/3,t)-v_n(2/3,t)\rVert\leq1/4$. 
Then 
\begin{align*}
&\lVert y(1/3,t)-y(2/3,t)\rVert\\
&\leq\lVert y(1/3,t)-v_n(1/3,t)\rVert
+\lVert v_n(1/3,t)-v_n(2/3,t)\rVert+\lVert v_n(2/3,t)-y(2/3,t)\rVert\\
&\leq3/4. 
\end{align*}
For $(s,t)\in[1/3,2/3]\times[0,\infty)$, we set 
\[
y(s,t)=y(1/3,t)\exp((3s{-}1)\log(y(1/3,t)^*y(2/3,t)))
\in U(B\otimes\mathcal{O}_\infty). 
\]
On the interval $[1/3,2/3]$, we have $\Lip(y(\cdot,t))<3\pi/2<24\pi$. 
By Lemma \ref{explog}, 
\[
\limsup_{t\to\infty}
\sup_{s\in[0,1]}\lVert[y(s,t),a\otimes1]\rVert
\leq243\limsup_{t\to\infty}\sup_{r\in[0,1]}\lVert[x(r,t),a]\rVert
\quad\forall a\in A. 
\]
Hence the proof is complete. 
\end{proof}

For a closed ideal $J$ of a unital $C^*$-algebra $A$, we denote 
$$U(J)=\{u\in U(A);\; u-1\in J\}.$$
In the next theorem, we replace $[0,1)$ with $[0,\infty)$ in the definitions of 
$A^\flat$ and $A_\flat$.  

\begin{theorem}\label{refinedNakamura}
Let $A$ be a unital separable $C^*$-algebra and let $J\subset A$ be a closed ideal. 
Let $(\alpha,u):G\curvearrowright A$ be a cocycle action 
of a countable discrete group $G$. 
Note that $\alpha$ induces a $G$-action on $A_\flat$, and we denote by 
$(A_\flat)^\alpha$ the fixed point subalgebra of $A_\flat$ for this action. 
Suppose that $(A_\flat)^\alpha$ contains 
a unital copy of $\mathcal{O}_\infty$. 
For any continuous map $x:[0,1]\times[0,\infty)\to U(J)$, 
there exists a continuous map 
$y:[0,1]\times[0,\infty)\to U(J)$ such that 
\[
y(0,t)=x(0,t),\quad y(1,t)=x(1,t),\quad \Lip(y(\cdot,t))<25\pi
\quad\forall t\in[0,\infty), 
\]
\[
\limsup_{t\to\infty}\sup_{s\in[0,1]}\lVert[y(s,t),a]\rVert
\leq243\limsup_{t\to\infty}\sup_{s\in[0,1]}\lVert[x(s,t),a]\rVert
\quad\forall a\in A
\]
and 
\[
\limsup_{t\to\infty}\sup_{s\in[0,1]}\lVert\alpha_g(y(s,t))-y(s,t)\rVert
\leq243\limsup_{t\to\infty}\sup_{s\in[0,1]}\lVert\alpha_g(x(s,t))-x(s,t)\rVert
\quad\forall g\in G. 
\]
\end{theorem}
\begin{proof}
Applying the proposition above to 
$J\subset A\rtimes_{(\alpha,u)}G$ and $x$, 
we get a continuous map 
$y:[0,1]\times[0,\infty)\to U(J\otimes\mathcal{O}_\infty)$ such that 
\[
y(0,t)=x(0,t)\otimes1,\quad y(1,t)=x(1,t)\otimes1,\quad 
\Lip(y(\cdot,t))\leq24\pi
\quad\forall t\in[0,\infty)
\]
and 
\[
\limsup_{t\to\infty}\sup_{s\in[0,1]}\lVert[y(s,t),a\otimes1]\rVert
\leq243\limsup_{t\to\infty}\sup_{s\in[0,1]}\lVert[x(s,t),a]\rVert
\quad\forall a\in A\rtimes_{(\alpha,u)}G. 
\]
Letting $\iota:\mathcal{O}_\infty\to(A_\flat)^\alpha$ be a unital embedding, 
we define a unital embedding 
\[
\phi:(A\rtimes_{(\alpha,u)}G)\otimes\mathcal{O}_\infty
\to(A\rtimes_{(\alpha,u)}G)^\flat
\]
by $\phi(a\otimes1)=a$ for $a\in A\rtimes_{(\alpha,u)}G$ and 
$\phi(1\otimes b)=\iota(b)$ for $b\in\mathcal{O}_\infty$. 
Then, we have 
\[
\phi(y(0,t))=x(0,t),\quad \phi(y(1,t))=x(1,t),\quad 
\Lip(\phi(y(\cdot,t)))\leq24\pi
\quad\forall t\in[0,\infty)
\]
and 
\[
\limsup_{t\to\infty}\sup_{s\in[0,1]}\lVert[\phi(y(s,t)),a]\rVert
\leq243\limsup_{t\to\infty}\sup_{s\in[0,1]}\lVert[x(s,t),a]\rVert
\quad\forall a\in A\rtimes_{(\alpha,u)}G. 
\]
Notice that $\phi(y(s,t))$ is in $U(J^\flat)$, 
because $J$ is an ideal of $A$. 

We would like to find a lift 
$\hat y:[0,1]\times[0,\infty)\to U(C^b([0,\infty),J))$ 
of $\phi(y(\cdot,\cdot)):[0,1]\times[0,\infty)\to U(J^\flat)$. 
For each $n\in\N\cup\{0\}$, 
the restriction of $\phi\circ y$ to $[0,1]\times[n,n{+}1]$ is regarded 
as an element of $U(C([0,1]\times[n,n{+}1])\otimes J^\flat)$, 
and so there exists a lift 
$\hat y_n:[0,1]\times[n,n{+}1]\to U(C^b([0,\infty),J))$. 
Inductively we define $\hat y$ by 
$\hat y(s,t)=\hat y_0(s,t)$ for $(s,t)\in[0,1]\times[0,1]$ and 
\[
\hat y(s,t)=\hat y(s,n)\hat y_n(s,n)^*\hat y_n(s,t)
\quad\forall (s,t)\in[0,1]\times(n,n{+}1]. 
\]
Then $\hat y:[0,1]\times[0,\infty)\to U(C^b([0,\infty),J))$ 
is a lift of $\phi\circ y$. 
We treat $\hat y$ as a continuous function 
in three variables $(s,t,r)\in[0,1]\times[0,\infty)\times[0,\infty)$. 

For any $t\in[0,\infty)$, we get 
\[
\lim_{r\to\infty}\hat y(0,t,r)=x(0,t),\quad 
\lim_{r\to\infty}\hat y(1,t,r)=x(1,t)
\]
and 
\[
\limsup_{r\to\infty}\Lip(\hat y(\cdot,t,r))<25\pi. 
\]
By a simple perturbation, 
we may assume that $\hat y(0,t,r)=x(0,t)$ and $\hat y(1,t,r)=x(1,t)$ hold 
for any sufficiently large $r>0$. 
Moreover, we have 
\[
\limsup_{t\to\infty}\sup_{s\in[0,1]}\limsup_{r\to\infty}
\lVert[\hat y(s,t,r)),a]\rVert
\leq243\limsup_{t\to\infty}\sup_{s\in[0,1]}\lVert[x(s,t),a]\rVert
\]
for all $a\in A\rtimes_{(\alpha,u)}G$. 
Hence, 
by choosing a rapidly increasing function $r:[0,\infty)\to[0,\infty)$ 
in a suitable way, one obtains 
\[
\limsup_{t\to\infty}\sup_{s\in[0,1]}\lVert[\hat y(s,t,r(t))),a]\rVert
\leq243\limsup_{t\to\infty}\sup_{s\in[0,1]}\lVert[x(s,t),a]\rVert
\]
for all $a\in A\rtimes_{(\alpha,u)}G$. 
Thus, the map $(s,t)\mapsto\hat y(s,t,r(t))$ does the job. 
\end{proof}


\end{document}